\documentclass[10pt]{amsart}
\usepackage{graphicx}
\usepackage{amscd}
\usepackage{amsmath}
\usepackage{amsthm}
\usepackage{amsfonts}
\usepackage{amssymb}
\usepackage{mathrsfs}
\usepackage{bm}
\usepackage{enumitem}
\usepackage{amsrefs}
\usepackage{xcolor}
\usepackage[colorlinks, citecolor=blue, linkcolor=red, pdfstartview=FitB]{hyperref}
\usepackage[latin1]{inputenc}

\newcommand{\bB}{{\mathbb{B}}} \newcommand{\BB}{ {\mathbb{B}}}
\newcommand{\bC}{{\mathbb{C}}} \newcommand{\CC}{ {\mathbb{C}}}

\newcommand{\bN}{{\mathbb{N}}} \newcommand{\NN}{ {\mathbb{N}}}

	\newcommand{\mc}[1]{\mathcal{#1}}
  \newcommand{\A}{{\mathcal{A}}}

  \newcommand{\F}{{\mathcal{F}}}
  
\renewcommand{\H}{{\mathcal{H}}}

  \newcommand{\M}{{\mathcal{M}}}
  
\renewcommand{\O}{{\mathcal{O}}}

\renewcommand{\S}{{\mathcal{S}}}

  \newcommand{\V}{{\mathcal{V}}}
  
  \newcommand{\X}{{\mathcal{X}}}
  
  \newcommand{\Z}{{\mathcal{Z}}}

\newcommand{\frk}[1]{\mathfrak{#1}}

\newcommand{\fa}{{\mathfrak{a}}}

\newcommand{\frb}{{\mathfrak{b}}}

\newcommand{\fc}{{\mathfrak{c}}}

\newcommand{\fH}{{\mathfrak{H}}}

\newcommand{\fK}{{\mathfrak{K}}}

\newcommand{\fM}{{\mathfrak{M}}}
\newcommand{\fm}{\mathfrak{m}}
\newcommand{\fN}{{\mathfrak{N}}}

\newcommand{\fp}{{\mathfrak{p}}}

\newcommand{\fR}{{\mathfrak{R}}}

\newcommand{\fv}{{\mathfrak{v}}}

\newcommand{\fX}{{\mathfrak{X}}}

\newcommand{\rC}{\mathrm{C}}

\newcommand{\eps}{\varepsilon}
\renewcommand{\phi}{\varphi}

\newcommand{\upchi}{{\raise.35ex\hbox{$\chi$}}}
\newcommand{\gra}{\alpha}

 \newcommand{\grD}{\Delta}

\newcommand{\grj}{\theta} 

\newcommand{\grl}{\lambda} \newcommand{\grL}{\Lambda}
 
\newcommand{\grw}{\omega} \newcommand{\grW}{\Omega}
\newcommand{\grz}{\zeta}

\newcommand{\ol}{\overline} \newcommand{\cc}[1]{\overline{#1}}

\newcommand{\qand}{\quad\text{and}\quad}

\newcommand{\diag}{\operatorname{diag}}

\newcommand{\ran}{\operatorname{ran}}
\newcommand{\re}{\operatorname{Re}}
\newcommand{\spn}{\operatorname{span}}

\newcommand{\fb}{\mathfrak{b}}

\newcommand{\card}{\operatorname{card }}

\newtheorem{lemma}{Lemma}[section]
\newtheorem{theorem}[lemma]{Theorem}

\newtheorem{corollary}[lemma]{Corollary}

\theoremstyle{definition}

\newtheorem{example}{Example}

\newcommand{\bksl}{\backslash}
\newcommand{\pd}{\partial}

\newcommand{\ZBd}{ \mc{Z}_{\mathbb{B}_d} }
\newcommand{\ip}[1]{\langle #1 \rangle}
\newcommand{\Ann}{\mathrm{Ann}}
\newcommand{\PI}{\frk{p}}

\begin{document}
\author{Rapha\"el Clou\^atre}

\address{Department of Mathematics, University of Manitoba, Winnipeg, Manitoba, Canada R3T 2N2}

\email{raphael.clouatre@umanitoba.ca\vspace{-2ex}}
\thanks{The first author was partially supported by an NSERC Discovery Grant.}
\author{Edward J. Timko}
\email{edward.timko@umanitoba.ca\vspace{-2ex}}

\begin{abstract}
We study similarity classes of commuting row contractions annihilated by what we call higher order vanishing ideals of interpolating sequences. Our main result exhibits a Jordan-type direct sum decomposition for these row contractions. We illustrate how the family of ideals to which our theorem applies is very rich, especially in several variables. We also give two applications of the main result. First, we obtain a purely operator theoretic characterization of interpolating sequences. Second, we classify certain classes of cyclic commuting row contractions up to quasi-similarity in terms of their annihilating ideals. This refines some of our recent work on the topic. We show how this classification is sharp: in general quasi-similarity cannot be improved to similarity. The obstruction to doing so is the existence, or lack thereof, of norm-controlled similarities between commuting tuples of nilpotent matrices, and we investigate this question in detail.
\end{abstract}

\title[Row contractions and interpolating vanishing ideals]{Row contractions annihilated by interpolating vanishing ideals}
\date{\today}
\maketitle


\section{Introduction}
Since the appearance of the seminal work of Sz.-Nagy and Foias in the original edition of \cite{nagy2010}, the rich interplay between complex function theory on the unit disc and the theory of Hilbert space contractions has been fruitfully exploited. A wealth of structural results about Hilbert space contractions was uncovered, based on a careful analysis of the compressions of the standard isometric unilateral shift to its coinvariant subspaces. At the root of this strategy is an important fact saying that ``almost coisometric" contractions (i.e pure contractions with one-dimensional defect spaces) are always unitarily equivalent to such compressions. Furthermore, the appropriate coinvariant subspace can be identified explicitly, and it encodes the ideal of holomorphic relations constraining the contraction.

Unfortunately, this condition on the contraction $T$ being almost coisometric is very rigid, and it is desirable to replace it with a more flexible one. A natural replacement is that $T$ should admit a cyclic vector. Perhaps surprisingly, such contractions can still be classified using compressions of the unilateral shift to a coinvariant subspace reflecting the holomorphic constraints satisfied by $T$. The compromise here is that the relationship between $T$ and its classifying model is weaker than unitary equivalence (or even similarity); it is usually referred to as \emph{quasi-similarity}. Despite this apparent weakness, several key pieces of information about $T$ can still be extracted using this scheme. In fact, this approach can be greatly expanded to move past the setting of cyclic vectors, and a genuine analogue of the Jordan canonical form of a matrix can be constructed (see \cite{bercovici1988} for a comprehensive treatment).

A modern trend in operator theory is to make the object of study a $d$-tuple of operators $T=(T_1,\ldots,T_d)$ on some Hilbert space $\fH$. It is natural, then, to aim to reproduce the very successful univariate program to elucidate the properties of \emph{row contractions}, that is $d$-tuples $T$ that are contractive when viewed as row operators from the $d$-fold direct sum $\fH^{(d)}$ to $\fH$. This has been carried out to a great extent by Popescu in a long series of papers starting with \cite{popescu1989},\cite{popescu1991}, where he shows that many aspects of the classical theory have close analogues in the multivariate context, where no commutativity is imposed on the operators $T_1,\ldots,T_d$.  In contrast, the structure of \emph{commuting} row contractions turns out to be more elusive. Nevertheless, building upon the groundwork laid in \cite{muller1993} and \cite{arveson1998}, a coherent theory has emerged in the last two decades, leveraging function theory on the so-called Drury-Arveson space to infer information about general commuting row contractions.

By way of analogy with the familiar univariate setting, it is then natural to wonder whether commuting row contractions can be classified using compressions of the Drury-Arveson shift to coinvariant subspaces. Classification up to unitary equivalence was achieved in \cite{arveson1998}, under the necessary condition that the defect space of the row contraction be one-dimensional. Recently, the authors have showed that the aforementioned more flexible quasi-similarity classification also has a satisfactory multivariate counterpart \cite{CT2018}.

So far, we described two different classifications for commuting row contractions: one up to unitary equivalence which requires strong conditions to be satisfied, and another up to quasi-similarity that is more widely applicable. There is another commonly used equivalence relation on linear operators that we have seemingly overlooked: similarity. The motivation behind this paper is thus the following question: what kind of commuting row contractions can be classified up to similarity using compressions of the Drury-Arveson shift to coinvariant subspaces? 

We note that this question has been considered in the single-variable case in \cite{clouatreSIM1},\cite{clouatreSIM2},\cite{clouatre2015jordan}. Interestingly, there are obstructions to similarity even in otherwise transparent cases. Indeed, for a pure cyclic contraction $T$ annihilated by a Blaschke product $\theta$ with distinct roots, it was shown in \cite{clouatreSIM1} that similarity between $T$ and the standard model operator $S_\theta$ is equivalent to the roots of $\theta$ forming a so-called interpolating sequence. Roughly speaking, the condition on the sequence being interpolating allows for the construction of enough commuting idempotents in the commutant of $T$, which can in turn be used to diagonalize $T$ up to similarity. Achieving diagonalization in our multivariate context is one of the main objectives of this paper, where the natural replacements for Blaschke products are vanishing ideals of zero sets, and germs thereof. In turn, we use the information we obtain on diagonalization to connect the similarity question for commuting row contractions to various function theoretic properties of the zero set. In particular, we characterize the property of a sequence being interpolating in purely operator theoretic terms. As another application, we refine the work done in \cite{CT2018} in some special cases. This refinement, and the limitations of it which we identify, lead us to a careful analysis of similarities between commuting tuples of nilpotent matrices. Controlling the norm of these similarities  is the salient feature of this endeavour, and as we illustrate, this task is much more complicated than what was witnessed in \cite{clouatreSIM1}. 

Let us now turn to describing the structure of the paper. In Section \ref{S:prelim} we introduce the necessary background material and notation, and gather some necessary preliminary results.  In Section \ref{S:ideals}, we study what we call higher order vanishing ideals of interpolating sequences and we show in Theorem \ref{T:CHVI} that this class of ideals is very rich, much more so in fact that its single-variable counterpart. This helps frame the main result of the paper on the existence of Jordan-type decompositions, which we prove in Section \ref{S:AJD}. In simplified terms, our main result reads as follows   (see Theorem \ref{T:Decomposition} for the complete statement). 
\begin{theorem}\label{T:mainA}
	Let $T=(T_1,\dots,T_d)$ be an absolutely continuous, commuting row contraction which is annihilated by all multipliers that vanish on some interpolating sequence $\Lambda\subset \bB_d$, up to some fixed order. Then, for each $\lambda\in \Lambda$ there is a commuting nilpotent $d$-tuple $N^{(\lambda)}$ such that $T$ is jointly similar to the $d$-tuple  $\oplus_{\lambda\in \Lambda}( \lambda I+N^{(\lambda)})$. Furthermore, the polynomials annihilating a given $d$-tuple $N^{(\lambda)}$ depend explicitly on the local behaviour of the annihilating ideal of $T$ at the point $\lambda$.
\end{theorem}

In the univariate situation, zero sets can be completely understood in terms of Blaschke products. No such tools are available in the multivariate world however. As a replacement tool, we perform a detailed analysis of germs of multipliers and of their polynomial representatives. This information is used crucially in the proof of Theorem \ref{T:mainA}.

The rest of the paper is devoted to two applications of our main result above. First, in Section \ref{S:siminterp}, we apply it to obtain the following purely operator theoretic characterization of interpolating sequences (Theorem \ref{T:siminterpequivalence}). This is a close multivariate analogue of \cite[Theorem 4.4]{clouatreSIM1}. 
Recall that a sequence $(\lambda_n)\subset \bB_d$ is \emph{strongly separated} if there is $\eps>0$ such that for every $n\in \bN$ there is a contractive multiplier $\omega_n$ such that $|\omega_n(\lambda_n)|\geq \eps$ and $\omega_n(\lambda_m)=0$ for every $m\neq n$.

\begin{theorem}\label{T:mainB}
	Let $\Lambda=\{\lambda_n:n\in \bN\}\subset \bB_d$ be a sequence and let $\fa$ denote its vanishing ideal of multipliers. Consider the following statements.
	\begin{enumerate}
		\item[\rm{(i)}] The sequence $\Lambda$ is interpolating.
		\item[\rm{(ii)}] The row contraction $Z^{\frk{a}}$ is similar to $D=\bigoplus_{n=1}^\infty \grl_n$, where $Z^{\frk{a}}$ is the compression of the Arveson $d$-shift to the orthogonal complement of $\fa$.
		\item[\rm{(iii)}] Every absolutely continuous commuting row contraction $T$ annihilated by $\fa$ is similar to $D=\oplus_{n=1}^\infty \lambda_n$. 
		\item[\rm{(iv)}] The sequence $\Lambda$ is strongly separated. 
		\item[\rm{(v)}] The sequence $\Lambda$ is strongly separated by partially isometric multipliers.
	\end{enumerate}
	Then {\rm (i)} $\Leftrightarrow$ {\rm (ii)} $\Leftrightarrow$ {\rm (iii)} $\Rightarrow$ {\rm (iv)} $\Leftrightarrow$ {\rm (v)}.
\end{theorem}

In one variable, the preceding statements are all equivalent, as a consequence of Carleson's classical characterization of interpolating sequences \cite[Chapter 9]{agler2002}. In several variables, the equivalence of (iv) and (v) appears to be new.
We also mention that the equivalence of (i) and (ii) is already apparent from results in \cite[Ch. 9]{agler2002}.

Our second application of Theorem \ref{T:mainA}  is a classification of certain pairs of commuting row contractions using their annihilating ideals. The following result (Theorem \ref{T:QSim}) gives a two-sided improvement of \cite[Corollary 3.7]{CT2018} in our special case of interest.
Specifically, this result says that $S$ and $T$ are \emph{quasi-similar}, which is a two-sided version of the notion of \emph{quasi-affine transform} that appears in \cite[Corollary 3.7]{CT2018}. 

\begin{theorem}\label{T:mainC}
	Let $T=(T_1,\cdots,T_d)$ and $S=(S_1,\cdots,S_d)$ be absolutely continuous, cyclic, commuting row contractions with common annihilating ideal $\frk{a}$.
	Assume that $\frk{a}$ contains some higher order vanishing ideal of an interpolating sequence.
	Then, there are injective operators $X$ and $Y$ with dense range such that $XT_k=S_kX$ and $YS_k=T_k Y$ for $1\leq k\leq d$.
\end{theorem}

We show in Example \ref{E:SimTrouble} that the previous theorem is sharp in the sense that quasi-similarity typically cannot be improved to similarity. Perhaps surprisingly, the obstruction lies in the structure of similarity classes of commuting tuples of nilpotent matrices. Elucidating this structure is a classical and notoriously difficult problem (see for instance \cite{GP1969},\cite{friedland1983}). This difficulty stands in sharp contrast with the case of a single cyclic nilpotent matrix, which of course is always similar to a Jordan block of appropriate size. Our point of view here is different however. In our case of interest, the existence of a similarity is easily established; it is the \emph{size} of this similarity that is crucial. In Section \ref{S:nilpotent}, we tackle this problem and obtain necessary and sufficient condition for the existence of norm-controlled similarities between a given nilpotent tuple and the corresponding functional model for a homogeneous annihilating ideal (see Theorems \ref{T:simnecessary} and \ref{T:sim}).

\section{Preliminaries}\label{S:prelim}
Throughout the paper, we fix a positive integer $d\geq 1$. We let $\fH$ denote a complex Hilbert space and $B(\fH)$ will denote the $\rC^*$-algebra of bounded linear operators on it. Likewise, we will denote by $B(\fH,\fK)$ the Banach space of bounded linear operators from $\fH$ into another Hilbert space $\fK$.  Given a subset $\S\subset B(\fH)$, we denote its commutant by $\S'$. We also set
\[
[\S\fH]=\ran \S=\ol{\spn\{A\xi:A\in \S,\xi\in \fH\}}.
\]
A $d$-tuple $T=(T_1,\ldots,T_d)$ of operators on $\fH$ is said to be \emph{cyclic} if there is a vector $\xi$ such that $[\A_T \xi]=\fH$, where $\A_T$ denotes the unital operator algebra generated by $T_1,\ldots,T_d$.
Given $z=(z_1,\cdots,z_d)\in\CC^d$ and $A\in B(\fH)$, we put
\[ z A=(z_1A,\cdots,z_dA). \]
Likewise, we set
\[ AT=(AT_1,\cdots,AT_d), \quad TA=(T_1A,\cdots,T_dA). \]
In particular, given another $d$-tuple $S=(S_1,\ldots,S_d)$ acting on some Hilbert space $\fK$, we say that $S$ and $T$ are \emph{quasi-similar} if there are injective bounded linear operators 
\[
X:\fH\to \fK, \quad Y:\fK\to \fH
\]
with dense ranges such that $XT=SX$ and $YS=TY$. If either $X$ or $Y$ is also surjective, then $S$ and $T$ are \emph{similar}.

\subsection{The Drury-Arveson space and interpolating sequences}\label{SS:DA}
Let $\bB_d\subset \bC^d$ denote the open unit ball. Define
\[
k(z,w)=\frac{1}{1-\langle z,w \rangle_{\bC^d}}, \quad z,w\in \bB_d.
\]
The associated reproducing kernel Hilbert space on $\bB_d$ is called the \emph{Drury-Arveson space} and it is denoted by $H^2_d$. We will encounter vector valued versions of this space as well, which we identify with $H^2_d\otimes \fH$, for some Hilbert space $\fH$. Given two Hilbert spaces $\fH$ and $\fK$, a function 
\[
\Phi:\bB_d\to B(\fH,\fK)
\]
 is a \emph{multiplier} if $\Phi f\in H^2_d\otimes \fK $ for every $f\in H^2_d\otimes \fH$. 
  
 Every multiplier $\Phi$ gives rise to a multiplication operator 
 \[
 M_\Phi: H^2_d\otimes\fH\to H^2_d\otimes\fK.
 \]
 Using this identification with $\fH=\fK=\bC$, we can view the algebra of multipliers as a weak-$*$ closed subalgebra $\M_d\subset B(H^2_d)$. One important property of $\M_d$ is that it coincides with its commutant, $\M_d=\M_d'$. We will often go back and forth between the interpretation of a multiplier as a function and as a multiplication operator. In particular, this identification allows us to define the \emph{multiplier norm} as
\[
\|\Phi\|=\|M_\Phi\|_{B(H^2_d\otimes\fH,H^2_d\otimes\fK)}
\]
for every multiplier $\Phi$. While the inequality
\[
\sup_{z\in \bB_d}\|\Phi(z)\|\leq \|\Phi\|
\]
holds for every multiplier $\Phi$, the two norms are not comparable in general. 

A multiplier $\Phi$ is \emph{inner} if $M_\Phi$ is a partial isometry. Much like in Beurling's classical description of invariant subspaces for the unilateral shift on the Hardy space of the unit disc, inner multipliers and ideals in $\M_d$ are connected with multiplier invariant subspaces in the Drury-Arveson space. We summarize the main features that we will need in the following theorem. \begin{theorem}\label{T:Beurling}
The following statements hold.
\begin{enumerate}
\item[\rm{(i)}] Let $\fN\subset H^2_d$ be a closed subspace which is invariant for $\M_d$. Then, there is a separable Hilbert space $\fH$ and an inner multiplier $\Omega:\bB_d\to B(\fH,\bC)$ such that $\fN=M_{\Omega} H^2_d(\fH)$.

\item[\rm{(ii)}] Let $\fa,\fb\subset \M_d$ be two weak-$*$ closed ideals. Then, $\fa=\fb$ if and only if $[\fa H^2_d]=[\fb H^2_d]$.
\end{enumerate}
\end{theorem}
\begin{proof}
(i) is \cite[Theorem 0.7]{MT2000} (see also \cite[Section 2]{arveson2002}), while (ii) is  \cite[Theorem 2.4]{DRS2015}.
\end{proof}

Standard calculations reveal that the coordinate functions $x_1,x_2,\ldots,x_d$ are multipliers of $H^2_d$, and that 
\[
M_{x_1}M_{x_1}^*+M_{x_2}M_{x_2}^*+\ldots+M_{x_d}M_{x_d}^*\leq I.
\]
In particular, the polynomials $\bC[x_1,\ldots,x_d]$ form a subset of $\M_d$, the norm closure of which we denote by $\A_d$. For every $w\in \bB_d$, there is a biholomorphic automorphism $\Gamma_w:\bB_d\to \bB_d$ such that $\Gamma_w(0)=w$ and 
\[
(\Gamma_w\circ \Gamma_w)(z)=\Gamma_0(z)=z, \quad z\in \bB_d.
\]
See \cite[Section 2.2]{rudin2008} for more details.
By \cite[Theorem 9.2]{davidsonramseyshalit2011} (see also \cite[Theorem 11.1.1]{shalit2014}), for every automorphism $\Gamma:\bB_d\to\bB_d$ there is a unitary $U_\Gamma\in B(H^2_d)$ such that 
\[
U_\Gamma \A_d U_\Gamma^*=\A_d
\]
and
\[
U_\Gamma M_\phi U_\Gamma^* =M_{\phi\circ \Gamma}.
\]
In particular, if we write
\[
\Gamma=(\gamma_1,\ldots,\gamma_d)
\]
where $\gamma_k:\bB_d\to \bC$ for $1\leq k\leq d$, then
\[
U_\Gamma M_{x_k} U_\Gamma^*=M_{\gamma_k}
\]
so that $\gamma_k\in \A_d$ and
\[
M_{\gamma_1}M_{\gamma_1}^*+M_{\gamma_2}M_{\gamma_2}^*+\ldots+M_{\gamma_d}M_{\gamma_d}^*\leq I.
\]
Given $\alpha=(\alpha_1,\alpha_2,\ldots,\alpha_d)\in \bN^d$, we put
\[
|\alpha|=\alpha_1+\alpha_2+\ldots+\alpha_d \qand \alpha!=\alpha_1!\alpha_2!\ldots\alpha_d!
\]
and we use the standard multi-index notations
\[
x^\alpha=x_1^{\alpha_1} x_2^{\alpha_2} \cdots x_d^{\alpha_d} \qand \frac{\pd^\alpha}{\pd x^\alpha}=\frac{\pd^{\alpha_1}}{\pd x_1^{\alpha_1}}\frac{\pd^{\alpha_2}}{\pd x_2^{\alpha_2}}\cdots \frac{\pd^{\alpha_d}}{\pd x_d^{\alpha_d}}.
\]
It can be verified that
\[
\|x^\alpha\|_{H^2_d}=\left(\frac{\alpha!}{|\alpha|!}\right)^{1/2}
\]
for every $\alpha\in \bN^d$.
Elements of $H^2_d$ are analytic functions $f:\bB_d\to\bC$ with the property that
\[
f(z)=\langle f, k_z \rangle_{H^2_d}, \quad z\in \bB_d
\]
where $k_z=k(\cdot,z)\in H^2_d$.  Now, for every $z\in \bB_d$ and $\alpha\in \bN^d$ we see that
\[
\frac{\pd^\alpha}{\pd \cc{z}^\alpha}\left(\frac{1}{1-\ip{x,z}}\right)=\frac{|\gra|!  x^\gra}{(1-\ip{x,z})^{|\gra|+1}}, \quad x\in \bB_d
\]
whence
\[
 \frac{\pd^\alpha k_z}{\pd \ol{z}^\alpha}\in H^2_d
\]
and
\begin{equation}\label{Eq:derivative}
\frac{\pd^\alpha f}{\pd x^\alpha}(z)=\left\langle f, \frac{\pd^\alpha k_z}{\pd \ol{z}^\alpha }\right\rangle_{H^2_d}
\end{equation}
for every $f\in H^2_d$. The reader should consult \cite{agler2002} for more details on these topics.

One basic property of the multiplier algebra $\M_d$ that we will need repeatedly is that it admits a solution to the so-called ``Gleason problem", as the following result shows.  

\begin{theorem}
	Let $\phi\in\mc{M}_d$ and let $z\in\BB_d$. Let $N\geq 1$ be an integer. For each $\alpha\in \bN^d$ such that $|\alpha|=N$ there is $\psi_\alpha\in \M_d$ with the property that
	\[ \phi=\sum_{|\alpha|<N}\frac{1}{\alpha!}\frac{\pd^\alpha \phi}{\pd x^\alpha}(z)(x-z)^\alpha+\sum_{|\alpha|=N}(x-z)^\alpha \psi_\alpha. \]
	\label{T:GleasonsTrick}
\end{theorem}
\begin{proof}
This can easily be inferred from \cite[Cor. 4.2]{gleason2005}.
\end{proof}

Let $\grL=\{\lambda_n:n\geq 1\}$ be a countable subset of $\BB_d$. Much of the developments in the paper are based on an analysis of the restrictions of multipliers to $\Lambda$, through the following properties. We say that $\Lambda$ is
\begin{enumerate}

\item \textit{separated} if for every $n\geq 1$ there is $\phi_n\in \M_d$ such that $\phi_n(\lambda_n)=1$ and
\[
\phi_n(\lambda_m)=0, \quad m\neq n,
\]

\item  \textit{strongly separated} if there is $\eps>0$ such that for every $n\geq 1$ there is $\omega_n\in \M_d$ with $\|\omega_n\|_{\M_d}\leq 1$ such that $|\omega_n(\lambda_n)|\geq \eps$ and
\[
\omega_n(\lambda_m)=0, \quad m\neq n.
\]

\item  \textit{strongly separated by inner multipliers} if there is $\eps>0$ such that for every $n\geq 1$ there is a Hilbert space $\fH_n$ and an inner multiplier $\Omega_n:\bB_d\to B(\fH_n,\bC)$ with $\|\Omega_n(\lambda_n)\|\geq \eps$ and
\[
\Omega_n(\lambda_m)=0, \quad m\neq n,
\]

\item an \textit{interpolating sequence for $\M_d$} if for every bounded sequence $(a_n)$ there is $\psi\in \M_d$ such that
\[
\psi(\lambda_n)=a_n, \quad n\geq 1.
\]
\end{enumerate}

Strongly separated sequences are obviously separated, and it is a consequence of the open mapping theorem that interpolating sequences are strongly separated. In one variable, the classical interpolation theorem of Carleson \cite{carleson1958} implies conversely that strongly separated sequences are strongly separated by inner functions, and in fact interpolating. However, this last implication fails for other reproducing kernel Hilbert spaces on the unit disc such as the Dirichlet space; this observation seems to be due to Bishop and Marshall--Sundberg \cite{MS1994}. In general, interpolating sequences can be characterized by another separation condition, along with a so-called Carleson measure condition. This important result can be found in \cite{AHMR2017}, and it settled a long-standing open problem.

\begin{theorem}\label{T:AHMR}
The countable subset $\grL=\{\lambda_n:n\geq 1\}\subset \bB_d$ is an interpolating sequence if and only if it satisfies the following two conditions
\begin{enumerate}

\item[\rm{(a)}] there is $\delta>0$ such that
\[
1-\frac{|k(\lambda_n,\lambda_m)|^2}{\|k_{\lambda_n}\|^2\|k_{\lambda_m}\|^2 }\geq \delta
\]
when $n$ and $m$ are distinct positive integers,

\item[\rm{(b)}]  there is $\gamma>0$ such that
\[
\sum_{n=1}^\infty \frac{|f(\lambda_n)|^2}{\|k_{\lambda_n}\|^2}\leq \gamma \|f\|^2
\]
for every $f\in H^2_d$.
\end{enumerate}
\end{theorem}

We record the following elementary consequence.

\begin{corollary}
	Let $\grL\subset \bB_d$ be an interpolating sequence.  Then, $\{w\}\cup\grL$ is also an interpolating sequence for every $w\in \bB_d$.
	\label{C:ISplus1}
\end{corollary}
\begin{proof}
We may clearly suppose that $w\notin \Lambda$. It is clear that $\{w\}\cup\grL$ satisfies property (b) in Theorem \ref{T:AHMR}, so it suffices to check that $\{w\}\cup\grL$ also satisfies property (a) therein. Assume otherwise, so that there is a subsequence $(\lambda_n)$ of $\Lambda$ with the property that
\[
\lim_{n\to\infty} \frac{|k(\lambda_n,w)|^2}{\|k_{\lambda_n}\|^2\|k_{w}\|^2 }=1.
\]
Upon passing to a further subsequence if necessary, we may assume that $(\lambda_n)$ converges to some $z\in \ol{\bB_d}$ and that the sequence of unit vectors
\[
\widehat{k_{\lambda_n}}=\frac{k_{\lambda_n}}{\|k_{\lambda_n}\|}, \quad n\geq 1
\]
converges weakly in $H^2_d$ to some $h\in H^2_d$ with $\|h\|\leq 1$. Note now that we have
\[
\left|\left\langle h,\frac{k_w}{\|k_w\|} \right\rangle\right|=\lim_{n\to\infty}\left|\left\langle \widehat{k_{\lambda_n}},\frac{k_w}{\|k_w\|} \right\rangle\right|^2=1
\]
so by the Cauchy-Schwarz inequality there is $\zeta\in \CC$ such that $|\zeta|=1$ and 
\[
h=\grz\frac{k_w}{\|k_w\|}.
\]
	Thus, we find
	\begin{align*}
	 \sqrt{1-\|w\|^2}&=|h(0)|=|\langle h,k_0\rangle|=\lim_{n\to\infty}|\langle \widehat{k_{\lambda_n}}, k_0\rangle|\\
	&=\lim_{n\to\infty} \sqrt{1-\|\lambda_n\|^2}=\sqrt{1-\|z\|^2}.
	\end{align*}
	whence $\|z\|=\|w\|<1$. Thus, the sequence $\Lambda$ has an accumulation point in the open unit ball $\bB_d$. Since subsequences of interpolating sequences are interpolating themselves, and since functions in $\M_d$ are continuous on $\bB_d$, this is readily seen to be impossible.
\end{proof}

\subsection{Multivariate operator theory and functional calculi}\label{SS:MOT}
Let  $T_1,\ldots,T_d\in B(\fH)$ be commuting operators, and let $T=(T_1,\ldots,T_d)$.  We will denote by $\sigma(T)\subset \bC^d$ the Taylor spectrum of $T$. See \cite{VasilescuBook} or \cite{muller2007} for comprehensive treatments of the Taylor spectrum.  One important tool we will need is the so-called \emph{Taylor functional calculus}. Given an open subset $U\subset \CC^d$, we let $\mc{O}(U)$ denote the ring of functions that are holomorphic on $U$. Then, there is a constant $C>0$ and a unital algebra homomorphism 
	\[
	\tau_{T,U}:\mc{O}(U)\to  \{T_1,\ldots,T_d\}''
	\]
	such that 
	\[
	\tau_{T,U}(x_j)=T_j, \quad 1\leq j\leq d
	\]
	and
	\[
	 \|\tau_{T,U}(f)\|\leq C\sup_{z\in\sigma(T)}|f(z)|
	 \]
	 for every $f\in \O(U)$ (see \cite[Theorem  III.9.9]{VasilescuBook}). The following summarizes the properties of the Taylor functional calculus that we will need.	 
	\begin{theorem}\label{T:Taylorprop}
	Let $T=(T_1,\ldots,T_d)$ be a commuting $d$-tuple of operators on some Hilbert space $\fH$. Let $U\subset \bC^d$ be an open set containing $\sigma(T)$. Let 
	\[
	\tau_{T,U}:\O(U)\to B(\fH)
	\]
	be the Taylor functional calculus. Then, the following statements hold.
	\begin{enumerate}
		\item[\rm{(i)}] If $V\subset \bC^d$ is another open set containing $\sigma(T)$, and if $f\in\mc{O}(U)$ and $g\in\mc{O}(V)$ are functions such that $f|_{U\cap V}=g|_{U\cap V}$, then
			\[ \tau_{T,U}(f)=\tau_{T,V}(g). \]
		\item[\rm{(ii)}]  Let $R=(R_1,\dots,R_d)$ be a commuting $d$-tuple of operators on a Hilbert space $\frk{K}$ with $\sigma(R)\subset U$.
			Let $X\in\mc{B}(\frk{H},\frk{K})$ be such that $XT_j=R_jX$ for $j=1,\dots,d$.
			Then
			\[ X\tau_{T,U}(f)=\tau_{R,U}(f)X \]
			for each $f\in \mc{O}(U)$.
		\item[\rm{(iii)}]  Let $K_1$ and $K_2$ be disjoint non-empty compact subsets of $\CC^d$, and suppose $U_1$ and $U_2$ are open disjoint neighbourhoods of $K_1$ and $K_2$, respectively.
			Let $\chi$ denote the characteristic function of the set $U_1$, and set
			$P=\tau_{T,U_1\cup U_2} (\chi). $
			Then, $P$ is a non-zero idempotent operator commuting with $T$ which satisfies
			\[ \sigma(T|_{\ran P})=K_1 \qand \sigma(T|_{\ran(I-P)})=K_2. \]
	\end{enumerate}
\end{theorem}
\begin{proof}
These facts can be found in Theorem III.13.5 along with Corollaries III.9.10 and III.9.11 of \cite[ ]{VasilescuBook}.
\end{proof}

Our attention will be focused on the subclass of commuting $d$-tuples $T=(T_1,\ldots,T_d)$ which are \emph{row contractions} in the sense that 
\[
T_1T_1^*+T_2T_2^*+\ldots+T_dT_d^*\leq I.
\]
A crucial example of a commuting row contraction is the \emph{Arveson shift}
\[
M_x=(M_{x_1},M_{x_2},\ldots,M_{x_d})
\]
which acts on the Drury-Arveson space $H^2_d$. Another important example is
\[
Z^{\fa}=P_{\H_\fa}M_x|_{\H_\fa}
\]
where $\fa\subset \M_d$ is any ideal and 
\[
\H_\fa=H^2_d\ominus [\fa H^2_d].
\]
It is readily verified that the constant function $1\in H^2_d$ is a cyclic vector for $M_x$, and that $P_{\H_\fa} 1$ is a cyclic vector for $Z^\fa$.

One reason which explains the importance of $M_x$ is that it plays a certain universal role among commuting row contractions, as we describe next. 
Let $T=(T_1,\ldots,T_d)$ be a commuting row contraction on some Hilbert space $\fH$. By \cite[Theorem 8.1]{arveson1998}, there is a unital completely contractive homomorphism
\[
\alpha_T:\A_d\to B(\fH)
\]
such that
\[
\alpha_T(x_j)=T_j,\quad 1\leq j\leq d.
\]
The commuting row contraction $T$ is said to be \emph{absolutely continuous} (or AC for short) if $\alpha_T$ extends to a weak-$*$ continuous unital algebra homomorphism on $\mc{M}_d$. It follows from \cite[Theorem 3.3]{CD2016duality} (see also \cite[Theorem 2.4]{CD2016abscont}) that $T$ is AC if and only if the sequence $(\alpha_T(\phi_n))$ converges to $0$ in the weak-$*$ topology of $B(\fH)$ whenever $(\phi_n)$ is a bounded sequence in $\A_d$ converging to $0$ pointwise on $\bB_d$. The latter two conditions are in fact equivalent to the sequence $(\phi_n)$ converging to $0$ in the weak-$*$ topology of $\M_d$.
Since the polynomial multipliers are weak-$*$ dense in $\M_d$, if $T$ is AC then the weak-$*$ continuous extension of $\alpha_T$ is unique and we denote it by 
\[
\widehat{\alpha_T}:\M_d\to B(\fH).
\]
The \emph{annihilating ideal} of $T$ is defined to be
\[
\Ann(T)=\{\psi\in\M_d:\widehat{\alpha_T}(\psi)=0\}.
\]
Next, we show that the functional calculus just defined is compatible with the Taylor functional calculus. This is folklore, but we provide the details for the reader's convenience.

\begin{theorem}
	Let $T=(T_1,\dots,T_d)$ be a commuting row contraction on some Hilbert space $\frk{H}$. Then, the following statements hold.
	\begin{enumerate}
		\item[\rm{(i)}] Let $U\subset \bC^d$ be an open set which contains the closed unit ball $\ol{\bB_d}$. Let $f\in\mc{O}(U)$ and put $\phi=f|_{\bB_d}$. Then $\phi\in \mc{A}_d$ and 
		\[
		\alpha_T(\phi)=\tau_{T,U}(f).
		\]
		\item[\rm{(ii)}]  Assume that $\sigma(T)\subset \BB_d$. Then, for every $\phi\in \A_d$ we have
		\[
		\alpha_{T}(\phi)=\tau_{T,\bB_d}(\phi).
		\]
		If in addition $T$ is AC, then for every $\psi\in \M_d$ we have
		\[
		\widehat{\alpha_T}(\psi)=\tau_{T,\bB_d}(\psi).
		\]
	\end{enumerate}
	\label{T:EqualFC}
\end{theorem}

\begin{proof}
	(i) Because  $M_x$ is a row contraction, we must have that $\sigma(M_x)\subset \cc{\BB}_d$, whence $\tau_{M_x,U}(f)$ is a well-defined element of $B(H^2_d)$ which lies in 
	\[
	\{M_{x_1},M_{x_2},\ldots,M_{x_d}\}''=\M_d''=\M_d.
	\] 
	It follows that there is a $\psi\in\mc{M}_d$ such that $\tau_{M_x,U}(f)=M_\psi$. On the other hand, since the function $f$ is holomorphic on a neighbourhood of the closed unit ball, upon expanding it as a convergent power series centred at the origin we find a sequence of polynomials $(p_n)$ that converge uniformly on a neighbourhood of $\ol{\bB_d}$ to $f$. Consequently, by the continuity property of $\tau_{M_x,U}$, we find that the sequence of operators
	\[
	\tau_{M_x,U}(p_n)=M_{p_n}, \quad n\geq 1
	\]
	converges in norm to
	\[
	\tau_{M_x,U}(f)=M_\psi.
	\]
	This forces $\psi\in \A_d$ and since the multiplier norm dominates the supremum norm over $\bB_d$, we also find that $(p_n(z))$ converges to $\psi(z)$ for every $z\in \bB_d$. In particular, we infer that $\psi=\phi$ so indeed $\phi\in \A_d$. Finally, using the continuity property of $\alpha_T$ and of $\tau_{T,U}$  we find that the sequence of operators
	\[
	M_{p_n}=\alpha_T(p_n)=\tau_{T,U}(p_n), \quad n\geq 1
	\]
	converges in norm to both $\alpha_T(\phi)$ and to $\tau_{T,U}(f)$, whence
	\[
	\alpha_T(\phi)=\tau_{T,U}(f).
	\]

	(ii) A standard polynomial approximation argument similar to the one used above shows that 
	\[
		\alpha_{T}(\phi)=\tau_{T,\bB_d}(\phi).
	\]
	for every $\phi\in \A_d$. Assume now that $T$ is AC and fix $\psi\in \M_d$. For each $n\geq 1$, we define $U_n$ to be the open ball of radius $(1-1/n)^{-1}$ centred at the origin, and we let 
	\[
	\psi_n(z)=\psi((1-1/n)z), \quad z\in U_n.
	\]
	Then, $\psi_n\in \O(U_n)$ so that $\psi_n\in \A_d$ and
	\[
	\alpha_T(\psi_n)=\tau_{T,U_n}(\psi_n)
	\]
	for every $n\geq 1$ by (i). Since $\sigma(T)\subset \bB_d$, we may invoke Theorem \ref{T:Taylorprop} to see that
	\[
	\widehat{\alpha_T}(\psi_n)=\alpha_T(\psi_n)=\tau_{T,U_n}(\psi_n)=\tau_{T,\bB_d}(\psi_n)
	\]
	for every $n\geq 1$.
	Furthermore, we note that by the uniform continuity of $\psi$ on $\sigma(T)$, we have that the sequence $(\psi_n)$ converges uniformly on $\sigma(T)$ to $\psi$, whence the sequence $(\tau_{T,\bB_d}(\psi_n))$ converges in norm to $\tau_{T,\bB_d}(\psi)$ by the continuity property of $\tau_{T,\bB_d}$. On the other hand, it is well-known that $(M_{\psi_n})_n$ converges to $M_\psi$ in the weak-$*$ topology \cite[Theorem 3.5.5]{shalit2014}, so that the sequence $(\widehat{\alpha_T}(\psi_n))$ converges in the weak -$*$ topology to $\widehat{\alpha_T}(\psi)$. Hence, the two limits must coincide and we conclude that
	\[
	\widehat{\alpha_T}(\psi)=\tau_{T,\bB_d}(\psi). \qedhere
	\]
	\end{proof}
In view of this result, we may unambiguously use the notation $\phi(T)$ to denote the functional calculus associated to a commuting row contraction $T$ and applied to a function $\phi$, provided that this makes sense to begin with. We will do so henceforth, and will not distinguish between the various functional calculi. 

\subsection{Analytic varieties, ideals and germs}\label{SS:varieties}

Let $U\subset \CC^d$ be an open set. Given a subset of functions $F\subset \mc{O}(U)$ and an open subset $W\subset U$, we put
\[ \mc{Z}_W(F)=\{z\in W: f(z)=0\text{ for every }f\in F\}. \]
If $F$ is a finite set $\{f_1,\cdots,f_m\}$, then we also write $\mc{Z}_W(f_1,\cdots,f_m)$ instead of $\Z_W(F)$.
Recall that a set $\V\subset U$ is an \textit{analytic variety} in $U$ if for every $z\in \V$ there is an open subset $W\subset U$ containing $z$ along with finitely many functions  $f_1,\cdots,f_m\in \mc{O}(W)$ such that
\[ \V\cap W=\mc{Z}_W(f_1,\cdots,f_m). \]

Next, fix $z\in \bC^d$. We define an equivalence relation on the set of functions that are holomorphic in a neighbourhood of $z$. Let $U_1,U_2\subset \bC^d$ be open neighbourhoods of $z$, and let $f_1\in \mc{O}(U_1)$ and $f_2\in\mc{O}(U_2)$ be functions such that $f_1|_{U_1\cap U_2}=f_2|_{U_1\cap U_2}$. In this case, we write $f_1\sim_z f_2$; the relation $\sim_z$ is an equivalence relation, and we denote by $[f]_z$ the equivalence class of a function $f$ holomorphic on a neighbourhood of $z$.
We call $[f]_z$ the \textit{germ of $f$ at $z$}, and we denote by $\mc{O}(z)$ the ring of all germs of holomorphic functions at $z$. 

In the following result, given a subset $S$ of a ring we denote by $\langle S\rangle$ the ideal generated by $S$.

\begin{theorem}
	Let $U\subset \bC^d$ be an open subset and let $F\subset \O(U)$. Then, $\Z_U(F)$ is an analytic variety in $U$. In fact,
	for each $z\in \mc{Z}_U(F)$ there is an open subset $W\subset U$ containing $z$ along with finitely many functions $f_1,\dots,f_k\in F$ such that
	\[ \mc{Z}_U(F)\cap W= \{z\in W: f_1(z)=\cdots=f_k(z)=0\} \]
	and
	\[ \ip{[f]_z:f\in F}=\ip{[f_1]_z,\cdots,[f_k]_z}. \]
	\label{T:GZZG} 
\end{theorem}
\begin{proof}
This follows from \cite[Theorem II.E.3]{GunningRossi} and its proof.
\end{proof}

Let $\bC[x_1,\ldots,x_d]$ be the ring of polynomials in $d$ variables. For each $z\in\CC^d$, let 
\[
\frk{m}_z=\ip{x_1-z_1,\cdots,x_d-z_d}
\]
which is a maximal ideal in $\bC[x_1,\ldots,x_d]$. Correspondingly, we let
\[
\widetilde{\fm_z}=\ip{[x_1-z_1]_z,\cdots,[x_d-z_d]_z}
\]
which is the maximal ideal in $\O(z)$.
The next sequence of lemmas shows, for some purposes, that polynomials, holomorphic functions and multipliers can be used interchangeably when studying germs.
First, we deal with a density question.

\begin{lemma}
	Let $\frk{c}$ be an ideal in $\mc{M}_d$, and let $\cc{\frk{c}}$ denote its weak-$*$ closure.
	Let $z\in\BB_d$. Assume that there is a positive integer $\mu$ such that
	\[
	\widetilde{\fm_z}^\mu\subset \langle [f]_z:f\in\frk{c}\rangle.
	\]
	Then, we have that
	\[ \langle[f]_z: f\in\cc{\frk{c}} \rangle = \langle [f]_z:f\in\frk{c} \rangle. \]
	\label{L:PIClosures}
\end{lemma}
\begin{proof}
	Let $\frk{a}=\ip{[f]_z:f\in\cc{\frk{c}} }$ and $\frk{b}=\ip{[f]_z:f\in\frk{c}}$, so that $\frk{b}\subset\frk{a}$. Let $N$ be the cardinality of the set $\{\alpha\in \bN^d:|\alpha|\leq \mu-1\}$ and
	consider the surjective linear map
	\[
	\Delta: \M_d\to \bC^N
	\]
	defined as
	\[
	\Delta(\psi)=\left(\frac{\pd^\alpha\psi}{\pd x^\alpha}(z)\right)_{|\alpha|\leq \mu-1}
	\]
	for every $\psi\in \M_d$. By virtue of Equation (\ref{Eq:derivative}), we see that $\Delta$ is weak-$*$ continuous, so that 
	\[
	\Delta(\ol{\fc})=\Delta(\fc).
	\]
	Fix $f\in \ol{\fc}$. By the previous equality, we find $g\in \fc$ with the property that $\Delta(f-g)=0$, whence 
	\[
	[f]_z-[g]_z\in \widetilde{\fm_z}^\mu \subset \fb.
	\]
	Finally, we find that
	\[
	[f]_z=[g]_z+([f]_z-[g]_z)\in \fb.
	\]
	We conclude that $\fa\subset \fb$ as desired.
	\end{proof}

Next, we introduce a mechanism to move between ideals of germs of holomorphic functions and ideals of polynomials.
Given an open subset $U\subset \bC^d$, a subset of functions $F\subset \O(U)$ and a point $z\in U$, we define the \textit{polynomial ideal determined by $F$ at $z$} to be
\[ \PI(F,z)=\{p\in\CC[x_1,\dots,x_d]: [p]_z\in \ip{[f]_z:f\in F}\}. \]
We now verify that that this construction yields nothing new if we start with an ideal of polynomials  $\fa\subset \bC[x_1,\ldots,x_d]$, provided that $\fa$ contains some power of the maximal ideal $\fm_z$.

\begin{lemma}
	Let $\frk{a}\subset \CC[x_1,\dots,x_d]$ be an ideal of polynomials, and let $z\in\CC^d$. 
	Assume that	there is a positive integer $\mu$ such that $\frk{m}_z^{\mu}\subset \frk{a}.$ Then, we have that
	$\PI(\frk{a},z)=\frk{a}. $
	\label{L:PIofPI}
\end{lemma}

\begin{proof}
	We trivially have that $\frk{a}\subset \PI(\frk{a},z)$. To prove the reverse inclusion, we fix $p\in \PI(\frk{a},z)$. By definition, this means that there are $p_1,\cdots,p_m\in \frk{a}$ and $[f_1]_z,\cdots,[f_m]_z\in\mc{O}(z)$ such that
	\[ [p]_z=\sum_{j=1}^m [p_j]_z[f_j]_z. \]
	Upon writing each $f_j$ as a power series convergent around $z$, we see that there is another function $g_j$ holomorphic near $z$ such that 
	$
	[g_j]_z\in \widetilde{\fm_z}^\mu
	$ 
	and a polynomial $r_j$ such that 
	\[
	[f_j]_z=[r_j]_z+[g_j]_z.
	\]
	In particular, we see that $p_jr_j\in \fa$ for every $1\leq j\leq m$ and thus
	\[
	\sum_{j=1}^m p_j r_j\in \fa.
	\]
	On the other hand, we see that
	\[ [p]_z-\sum_{j=1}^m [p_j]_z[ r_j]_z=\sum_{j=1}^m [p_j]_z[g_j]_z\in \widetilde{\fm_z}^\mu, \]
	and since $p-\sum_{j=1}^m p_j r_j$ is a polynomial, this means that
	\[
	p-\sum_{j=1}^m p_j r_j\in  \frk{m}_z^{\mu}\subset \fa.
	\]
	Finally, we see that
	\[
	p=\left( p-\sum_{j=1}^m p_j r_j\right)+\sum_{j=1}^m p_j r_j\in \fa. \qedhere
	\]
\end{proof}

The next result shows that if $z$ is an isolated point of $\Z_U(F)$, then the polynomial ideal determined by $F$ at $z$ contains all the relevant information about $F$.

\begin{lemma}
	Let $U\subset\bC^d$ be an open subset and let $z\in U$. Let $F\subset \mc{O}(U)$ be a subset with the property that  $z$ is an isolated point of $\Z_U(F)$. Then
	\[ \ip{[p]_z:p\in\frk{p}(F,z)} = \ip{[f]_z:f\in F}. \]
	Furthermore, the radical of $\frk{p}(F,z)$ is $\frk{m}_z$, and there is a positive integer $\mu$ such that $\fm_z^\mu\subset \fp(F,z)$.
	\label{L:GetPIs}
\end{lemma}
\begin{proof}
	For convenience, throughout the proof we let
	\[
	\fa=\ip{[f]_z:f\in F}.
	\]
	Plainly, $\ip{[p]_z:p\in\frk{p}(F,z)}\subset \frk{a}$. To establish the converse, we first make a preliminary observation.
	Because $z$ is an isolated point of $\Z_U(F)$, it follows from the Nullstellensatz for $\mc{O}(z)$ (see \cite[Theorems II.E.20 and III.A.7]{GunningRossi}) that the radical of $\frk{a}$ is $\widetilde{\fm_z}$.	In particular, for each $1\leq j\leq d$ there is a positive integer $\mu_j$  such that $[(x_j-z_j)^{\mu_j}]_z\in \frk{a}$. Taking \
	\[
	\mu=\mu_1+\mu_2+\ldots+\mu_d
	\]
	we see that $[(x-z)^\alpha]_z\in \frk{a}$ whenever $|\alpha|\geq \mu$, so that 
	\[
	\langle[q]_z:q\in \frk{m}_z^\mu\rangle\subset \fa
	\]
	and $ \frk{m}_z^\mu\subset \frk{p}(F,z)$. This immediately implies that the radical of $ \frk{p}(F,z)$ contains $\fm_z$, and hence is equal to $\fm_z$ by maximality. Fixing now $[f]_z\in \fa$, we will show that $[f]_z\in\ip{[p]_z:p\in\frk{p}(F,z)}$, thus completing the proof. Upon writing $f$ as a power series convergent around $z$, we see that there is another function $g$ holomorphic near $z$ such that 
	\[
	[g]_z\in \ip{[q]_z:q\in \frk{m}_z^\mu}
	\]
	and a polynomial $r$ such that 
	\[
	[f]_z=[r]_z+[g]_z.
	\]
	It thus suffices to show that $[r]_z$ and $[g]_z$ belong to $\ip{[p]_z:p\in\frk{p}(F,z)}$. Using that 
	$
	\{[q]_z:q\in \frk{m}_z^\mu\}\subset \fa
	$
	we infer  $[g]_z\in \frk{a}$, so it follows that 
	\[
	[r]_z=[f]_z-[g]_z\in\frk{a}
	\]
	and therefore $r\in\frk{p}(F,z)$ by construction of $\frk{p}(F,z)$, so we indeed have
	$
	[r]_z\in\ip{[p]_z:p\in\frk{p}(F,z)}.
	$
	On the other hand, we have $\frk{m}_z^\mu\subset\frk{p}(F,z)$, so using that  $[g]_z\in \langle [q]_z:q\in \frk{m}_z^\mu\rangle$ we find
	$
	[g]_z\in \ip{[p]_z:p\in\frk{p}(F,z)}
	$
	as desired. We have thus shown that
	$
	\ip{[p]_z:p\in\frk{p}(F,z)}=\fa.
	$
	\end{proof}

The previous lemma allows us to make an important definition that we require later. Let $U\subset\bC^d$ be an open subset and let $z\in U$. Let $F\subset \mc{O}(U)$ be a subset with the property that  $z$ is an isolated point of $\Z_U(F)$. By Lemma \ref{L:GetPIs}, there is a positive integer $\mu$ such that $\fm_z^\mu\subset \frk{p}(F,z)$. We may thus define the \textit{polynomial order} of $F$ at $z$ to be the smallest positive integer $\kappa$ such that
 $\fm_z^{\kappa+1}\subset \frk{p}(F,z)$. In the special case where $\Z_U(F)$ is a discrete subset of $U$, then $F$ has a well-defined polynomial order at every $z\in \Z_U(F)$.

\section{Higher order vanishing ideals}\label{S:ideals}

In this section, we consider higher order vanishing ideals of interpolating sequences. In subsequent sections, these objects will form the basis of various operator theoretic problems. For now, we construct such ideals, and show that the multivariate setting supports a wealth of drastically different behaviours, making it much richer and more complicated than the familiar univariate situation. In turn, this provides motivation for the work appearing in later sections.

Throughout this section, $\Lambda\subset \bB_d$ will be a countable set.  For each non-negative integer $\kappa$, we define the \emph{vanishing ideal of $\Lambda$ of order $\kappa$}, denoted by $\fv_\kappa(\Lambda)$, to be the collection of functions $\psi\in \M_d$ such that
\[
\frac{\pd^\alpha \psi}{\pd x^\alpha}(z)=0
\]
for every $z\in \Lambda$ and every $\alpha\in \bN^d$ such that $|\alpha|\leq \kappa$. It follows from Equation (\ref{Eq:derivative}) that $\fv_\kappa(\Lambda)$ is a weak-$*$ closed ideal of $\M_d$. Our attention will be mostly devoted to the case where $\Lambda$ is an interpolating sequence (see Subsection \ref{SS:DA}). In this case, we can identify the zero set of $\fv_\kappa(\Lambda)$.

\begin{theorem}
	Let $\grL\subset \bB_d$ be an interpolating sequence and let $\kappa$ be a non-negative integer. Then, we have that 
	$
	\Lambda=\Z_{\bB_d}(\fv_\kappa(\Lambda)).
	$
	\label{T:ISDim0}
\end{theorem}
\begin{proof}
It is immediate from the definition of $\fv_\kappa(\Lambda)$ that $\Lambda\subset \Z_{\bB_d}(\fv_\kappa(\Lambda))$. Conversely, let $w\in\BB_d\setminus \grL$ and invoke Corollary \ref{C:ISplus1} to find $\theta\in\M_d$ vanishing on $\Lambda$ such that $\theta(w)=1$. Then, we see that $\theta^{\kappa+1}\in \fv_\kappa(\Lambda)$ so $w\notin \Z_{\bB_d}(\fv_\kappa(\Lambda))$ and the proof is complete.
\end{proof}

The class of ideals we will be interested in for the rest of the paper are those that contain $\fv_\kappa(\Lambda)$ for some $\kappa$; we will typically refer to them as \emph{higher order vanishing ideals}. One useful property that such ideals possess is that they have uniformly bounded polynomial order at every point in their zero set.

\begin{lemma}
	Let $\grL\subset \bB_d$ be an interpolating sequence and let $\kappa$ be a non-negative integer. Let $\frk{a}\subset \mc{M}_d$ be an ideal  such that $\frk{a}\supset \frk{v}_\kappa(\grL)$. Then, for every $z\in \Z_{\bB_d}(\fa)$,  the polynomial order of $\fa$ at $z$ is at most $\kappa$.
	\label{L:FinVanOrd}
\end{lemma}
\begin{proof}
	Fix $z\in\grL$ and let $\theta\in\mc{M}_d$ be a multiplier vanishing on $\Lambda\setminus \{z\}$ and such that $\theta(z)=1$.
	Let $\alpha\in\NN^d$ such that $|\alpha|>k$. Then, it is readily verified that $(x-z)^\alpha\theta^{\kappa+1}\in \fv_\kappa(\Lambda)$, so that $(x-z)^\alpha\theta^{\kappa+1}\in \fa$ and  $[(x-z)^\alpha\theta^{\kappa+1}]_z\in \ip{[f]_z:f\in\frk{a}}$. Next, the fact that $\theta(z)=1$ implies that $[\theta]_z$ is invertible in $\O(z)$, so it follows that $(x-z)^\alpha\in\frk{p}(\frk{a},z)$. This shows that $\fm_z^{\kappa+1}\subset \fp(\fa,z)$.
\end{proof}

Next, we aim to elucidate the structure of higher order vanishing ideals.
For this purpose, it is useful to first consider the single-variable case.

\begin{example}\label{E:d=1vanishing}
Let $\Lambda\subset \bB_1$ be an interpolating sequence, let $\kappa$ be a non-negative integer and let $\fa\subset \M_1$ be a weak-$*$ closed ideal containing $\fv_\kappa(\Lambda)$. Let $\theta\in \M_1$ denote the Blaschke product with a simple zero at every point of $\Lambda$. Using the classical inner-outer factorization along with the factorization of inner functions as Blaschke products and singular inner functions, it is readily seen that the fact that $\fa$ contains $\fv_\kappa(\Lambda)$ is equivalent to $\theta^{\kappa+1}\in \fa$.
There is an inner function $\tau \in \M_1$ such that $\fa=\tau \M_1$. Then, we see that $\tau$ divides $\theta^{\kappa+1}$, so that $\tau$ is itself a Blaschke product with
\[
\Z_{\bB_1}(\tau)=\Z_{\bB_1}(\fa)\subset \Z_{\bB_1}(\fv_\kappa(\Lambda))=\Lambda
\]
 where the last equality follows from Theorem \ref{T:ISDim0}. Write
\[
\tau(x)=\prod_{z\in \Z_{\bB_1}(\fa)} \left(\frac{\ol{z}}{|z|} \frac{x-z}{1-\ol{z}x}\right)^{n_z}, \quad x\in \bB_1
\]
where each $n_z$ is a positive integer at most $\kappa+1$. It is then readily verified that 
\begin{equation}\label{Eq:d=1vanishing}
\fm_z^{n_z}=\fp(\fa,z)
\end{equation}
for every $z\in \Z_{\bB_1}(\fa)$. 
\qed
\end{example}

Our next task is to show that the simple behaviour witnessed in the previous example is special to the univariate setting. Some preparation is required.

\begin{lemma}\label{L:GetIdems}
	Let $\grL\subset\bB_d$ be an interpolating sequence and let $\kappa$ be a non-negative integer.  For each subset $\Omega\subset \Lambda$, there is a multiplier $\theta_\Omega\in \M_d$ with the following properties.
	\begin{enumerate}
	
	\item[\rm{(i)}] For each subset $\Omega\subset \Lambda$, we have $\theta_\Omega\in \fv_\kappa(\Lambda\setminus \Omega)$ and $1-\theta_\Omega\in \fv_\kappa(\Omega)$.
	
	\item[\rm{(ii)}] We have
	\[
	\sup_{\Omega\subset \Lambda}\|\theta_\Omega\|<\infty.
	\]
	
	\item[\rm{(iii)}] The ideal
	\[
	\fv_\kappa(\Lambda)+\sum_{z\in\grL} \theta_{\{z\}}\mc{M}_d
	\]
	is weak-$*$ dense in $\mc{M}_d$.
	\end{enumerate}
 \end{lemma}
\begin{proof}
	Since $\Lambda$ is an interpolating sequence, an application of the open mapping theorem yields the existence of a constant $C>0$ such that for every bounded function $f$ on $\Lambda$, there is a corresponding multiplier in $\M_d$ whose restriction to $\Lambda$ coincides with $f$ and whose norm is at most
	\[
	C\sup_{z\in \Lambda}|f(z)|.
	\]
	In particular, for every $\Omega\subset \Lambda$ there is $\phi_\Omega\in \M_d$ whose restriction to $\Lambda$ agrees with the characteristic function of $\Omega$ and such that $\|\phi_\Omega\|\leq C$. For each $\Omega\subset \Lambda$,  we put
	\[ \grj_\grW = 1-(1-\phi_\grW^{\kappa+1})^{\kappa+1}. \]
	It is readily checked that these functions satisfy properties (i) and (ii).
	To establish (iii), let $(\Omega_n)$ be an increasing sequence of finite subsets of $\grL$ such that $\cup_{n=1}^\infty \grW_n=\grL$.
	For each positive integer $n$, we claim that the multiplier
	\[
	\psi_n=\grj_{\grW_n}-\sum_{z\in \grW_n}\grj_{\{z\}}
	\]
	belongs to $\fv_\kappa(\Lambda).$
	To see this, first note that for every $\Omega\subset \Lambda$ we have $\theta_\Omega\in \fv_\kappa(\Lambda\setminus \Omega)$ and $1-\theta_\Omega\in \fv_\kappa(\Omega)$, which implies in particular that
	\begin{equation}\label{Eq:deriv0}
	\frac{\pd^\alpha \grj_{\Omega}}{\pd x^\alpha}(w)=0
	\end{equation}
	for every $w\in \Lambda$ and every $\alpha\in \bN^d$ such that $|\alpha|\geq 1$ and $|\alpha|\leq \kappa$.
	Fix $w\in \Lambda$ and $n\geq 1$. If $w\in \Omega_n$, then we have
	\begin{align*}
	\grj_{\grW_n}(w)-\sum_{z\in \grW_n}\grj_{\{z\}}(w)&=\grj_{\grW_n}(w)-\grj_{\{w\}}(w)\\
	&=1-1=0
	\end{align*}
	while if $w\in\Lambda\bksl \Omega_n$ then we have
	\begin{align*}
	\grj_{\grW_n}(w)-\sum_{z\in \grW_n}\grj_{\{z\}}(w)&=0-0=0.
	\end{align*}
	Furthermore, it follows from Equation (\ref{Eq:deriv0}) that
	\begin{align*}
	\frac{\pd^\alpha}{\pd x^\alpha}\left(\grj_{\grW_n}-\sum_{z\in \grW_n}\grj_{\{z\}} \right)(w)&=0
	\end{align*}
	for every $w\in\Lambda$ and every $\alpha\in \bN^d$ such that $|\alpha|\geq 1$ and $|\alpha|\leq \kappa$.  This establishes the claim that $\psi_n\in \fv_\kappa(\Lambda)$.
	Now, the sequence $(\grj_{ \grW_n})$ is bounded, and hence it has a weak-$*$ limit point $\tau\in \M_d$. Note that $\tau$ is the weak-$*$ limit of
	\[
	\psi_n+\sum_{z\in \grW_n}\grj_{\{z\}}
	\]
	whence $\tau$ lies in the weak-$*$ closure of $\fv_\kappa(\Lambda)+\sum_{z\in\grL} \theta_{\{z\}}\mc{M}_d$.
	On the other hand, for each $z\in \Lambda$, there is $N\geq 1$ such that $z\in \Omega_n$ for every $n\geq N$.  We thus find
	\[
	\tau(z)=\lim_{n\to\infty} \theta_{\Omega_n}(z)=1
	\]
 	and using Equation (\ref{Eq:derivative}) we obtain
	\[
	\frac{\pd^\alpha \tau}{\pd x^\alpha}(z)=\lim_{n\to\infty} \frac{\pd^\alpha \theta_{ \Omega_n}}{\pd x^\alpha}(z)=0
	\]
	for $z\in \Lambda$ and every $\alpha\in \bN^d$ such that $|\alpha|\geq 1$ and $|\alpha|\leq \kappa$, whence $(1-\tau)\in \fv_\kappa(\Lambda)$.  We conclude that $1=(1-\tau)+\tau$ lies in the weak-$*$ closure of
	$\fv_\kappa(\Lambda)+\sum_{z\in\grL} \theta_{\{z\}}\mc{M}_d$, so that this ideal is indeed weak-$*$ dense.
	\end{proof}

We now arrive at the main result of this section, showing that higher order vanishing ideals are plentiful.

\begin{theorem}
	Let $\grL\subset \bB_d$ be an interpolating sequence and let $\kappa$ be a non-negative integer. For each $z\in \Lambda$, let $\fa_z\subset \CC[x_1,\dots,x_d]$ be an ideal containing $\fm_z^{\kappa+1}$. Then, there exists a weak-$*$ closed ideal $\frk{b}\subset \M_d$ containing $\fv_\kappa(\Lambda)$ with the property that $\ZBd(\frk{b})=\grL$ and
	$ \PI(\frk{b},z)=\frk{a}_z  $
	for each $z\in\grL$.
	\label{T:CHVI} 
\end{theorem}
\begin{proof}
	By Lemma \ref{L:GetIdems}, for each $z\in\Lambda$ there is a multiplier $\theta_z\in \fv_\kappa(\Lambda\setminus \{z\})$ such that $\theta_z(z)=1$. Since $\fm_z^{\kappa+1}\subset \fa_z$, we see that $\mc{Z}_{\CC^d}(\frk{a}_z)=\{z\}$ and infer $\theta_z\fa_z\subset \fv_0(\Lambda)$ for every $z\in \Lambda$. In particular, we see that the ideal
	\[
	\frk{b}_0=\sum_{z\in \Lambda} \theta_z \fa_z 
	\]
	is contained in $\fv_0(\Lambda)$. Let $\frk{b}$ be the weak-$*$ closure of $\frk{b}_0+\frk{v}_\kappa(\grL)$ in $\M_d$. In view of the inclusions
	\[ \frk{v}_\kappa(\grL)\subset\fb\subset \fv_0(\Lambda) \] and of Theorem \ref{T:ISDim0}, we see that $\ZBd(\frk{b})=\grL$. It only remains to show that $ \PI(\frk{b},z)=\frk{a}_z  $
	for each $z\in\grL$. 
	
	For this purpose, fix $z\in \Lambda$. We claim that 
	\[
	\ip{[\psi]_z:\psi\in\frk{b}_0}=\ip{[p]_z:p\in\frk{a}_z}.
	\]
	Indeed, for every $w\in \Lambda, w\neq z$ and every $\alpha\in \bN^d$ with $|\alpha|\leq \kappa$, we note that 
	\[
	\frac{\pd^\alpha \theta_w}{\pd x^\alpha}(z)=0
	\]
	so that 
	\[
	[\theta_w]_z \in \ip{ [q]_z:q\in \fm_z^{\kappa+1}}.
	\]
	On the other hand, we know by assumption that $ \fm_z^{\kappa+1}\subset\fa_z$ so
	\[
	[\theta_w]_z \in \ip{ [p]_z:p\in \fa_z}
	\]
	for every $w\in \Lambda, w\neq z$, and therefore
	\begin{equation}\label{Eq:inclusion}
	\langle[\psi]_z:\psi\in\fb_0\rangle\subset \langle [\psi]_z:\psi\in \fa_z+\theta_z\fa_z\rangle.
	\end{equation}
	Since $\theta_z(z)=1$, we conclude that $[\theta_z]_z$ is invertible in $\O(z)$. Consequently, we find
	\begin{align*}
	\langle [p]_z:p\in \fa_z\rangle&=\langle [\theta_z p]_z:p\in \fa_z\rangle\\
	&\subset \langle[\psi]_z:\psi\in\fb_0\rangle\\
	&\subset \langle [p]_z:p\in \fa_z\rangle
	\end{align*}
	where the last inclusion follows from Equation (\ref{Eq:inclusion}). The claim is established.
	Next, we always have
	\[
	\ip{[\psi]_z:\psi\in\frk{v}_k(\grL)}\subset \ip{[q]_z:q\in\frk{m}_z^{k+1}}
	\]
	so in particular
	\[
	\ip{[\psi]_z:\psi\in\frk{v}_k(\grL)}\subset \ip{[p]_z:p\in\fa_z}
	\]
	and thus
	\[ \ip{[\psi]_z:\psi\in\frk{b}_0+\frk{v}_k(\grL)}=\ip{[\psi]_z:\psi\in\frk{b}_0 }=\ip{[p]_z:p\in\frk{a}_z}. \]
	By Lemma \ref{L:PIClosures}, it follows that
	\[ \ip{[\psi]_z:\psi\in\frk{b}}=\ip{[p]_z:p\in\frk{a}_z} \]
	and so $\frk{p}(\frk{b},z)=\frk{p}(\frk{a}_z,z)$. Finally, we may invoke Lemma \ref{L:PIofPI} to get
	$
	\fp(\fb,z)=\fa_z.
	$
\end{proof}

This theorem says in particular that the higher order vanishing behaviour of multipliers in several variables is more complicated that what is visible on the unit disc.  We illustrate this in the following example, leveraging the fact that we can prescribe the polynomial ideals determined by such an ideal at every point.

\begin{example}\label{E:vanishingideal}
Let $\Lambda=\{z_n:n\in \bN\}\subset \bB_2$ be an interpolating sequence. For each $n\in \bN$, write
$
z_n=(z_{n,1},z_{n,2})
$
 and consider the ideal
\[
	\fb_n=\fm_{z_n}^2+\ip{x_1-z_{n,1}}.
\]
It is readily verified that $\fb_n$ is not a power of the maximal ideal $\fm_{z_n}$. By Theorem \ref{T:CHVI}, there are weak-$*$ closed ideals $\fa,\fb\subset \M_2$ both containing $\fv_2(\Lambda)$ and satisfying $\Z_{\bB_d}(\fa)=\Z_{\bB_d}(\fb)=\Lambda$ such that
\[
\fp(\fa,z_n)=\fm_{z_n}^2, \quad \fp(\fb,z_n)=\fb_n
\]
for every $n\geq 1$. This type of phenomenon does not occur in one variable; the reader may wish to compare the preceding equalities with Equation (\ref{Eq:d=1vanishing}) in Example \ref{E:d=1vanishing}.
\qed
\end{example}

\section{Jordan-type decompositions}\label{S:AJD}

In this section, we investigate AC commuting row contractions annihilated by some ideal of multipliers. Our main goal is to show that if the annihilating ideal is a higher order vanishing ideal for some interpolating sequence, then the corresponding commuting row contraction is similar to a block diagonal tuple, where each block is a nilpotent tuple translated by a scalar multiple of the identity. We view this as an infinite-dimensional, multivariate version of the classical Jordan decomposition of a matrix. Several preliminary lemmas are required before we can prove the existence of such a decomposition. We start with a technical fact.

\begin{lemma}
Let $T=(T_1,\dots,T_d)$ be an AC commuting row contraction on some Hilbert space $\frk{H}$. Let $\S\subset \M_d$ be a subset with the property that $\S+\Ann(T)$ is weak-$*$ dense in $\mc{M}_d$.  Then, we have
	\[
	\frk{H}=\bigvee_{\phi\in \S} \ran \phi(T).
	\]
	\label{L:TotalRange}
\end{lemma}
\begin{proof}
Throughout the proof we put
\[
\fR=\bigvee_{\phi\in \S} \ran \phi(T).
\]
	We note that
	\[ \fR^\perp = \bigcap_{\phi\in \S} \ker \phi(T)^*. \]
	By assumption, the constant multiplier $1$ is in the weak-$*$ closure of $\S+\Ann(T)$.
	Thus there is a net $(\psi_j)_{j\in J}$ in $\Ann(T)$ and a net $(\phi_j)_{j\in J}$ in $\S$ such that $(\phi_j+\psi_j)_{j\in J}$ converges in the weak-$*$ topology of $\M_d$ to $1$.
	On the other hand, since $\psi_j(T)=0$ for every $j\in J$ and since $T$ is AC, we have that the net $(\phi_j(T))_{j\in J}$ converges to $I$ in the weak-$*$ topology of $B(\fH)$. For $\xi \in  \bigcap_{\phi\in \S} \ker \phi(T)^*$, we find
	\[ \ip{\xi,\xi}=\lim_j \ip{\xi,\phi_j(T)\xi} = \lim_j \ip{\phi_j(T)^*\xi,\xi} = 0 \]
	which implies $\fR^\perp=\{0\}$ as desired.
\end{proof}

We now clarify the relationship between the Taylor spectrum and the zero sets of annihilating ideals.

\begin{theorem}
	Let $T=(T_1,\dots,T_d)$ be an AC commuting row contraction. Then, the following statements hold.
	\begin{enumerate}
	
	\item[\rm{(i)}] We have that
	\[ \sigma(T)\cap\BB_d \subset \ZBd(\Ann (T)). \]
	
	\item[\rm{(ii)}] 	Let $\grL\subset \bB_d$ be an interpolating sequence and let $\kappa$ be a non-negative integer. Assume that $\fv_\kappa(\Lambda)\subset \Ann(T)$.
	Then, we have that
	\[
	\fv_\kappa(\sigma(T)\cap \bB_d)\subset \Ann(T)
	\]
	and 
	\[ \sigma(T)\cap\BB_d=\ZBd(\Ann (T)). \]
	\end{enumerate}
	\label{T:spectrum}
\end{theorem}
\begin{proof}
	(i) Let $z\in \BB_d\setminus \ZBd(\Ann (T))$ and choose $\phi\in \Ann(T)$ such that $\phi(z)=1$.
	By Theorem \ref{T:GleasonsTrick} there are $\psi_1,\dots,\psi_d\in\mc{M}_d$ such that 
	\[
	\phi=1+\sum_{j=1}^d(x_j-z_j)\psi_j.
	\]
	Applying the functional calculus to the previous equality yields
	\[ 0=\phi(T)=I+\sum_{j=1}^d(T_j-z_jI)\psi_j(T) \]
	which forces $z\notin \sigma(T)$ by \cite[Proposition IV.25.3]{muller2007}. 
	
	(ii) By Lemma \ref{L:GetIdems}, for each subset $\Omega\subset \Lambda$ there is a multiplier $\theta_\Omega\in \fv_\kappa(\Lambda\setminus \Omega)$ such that $1-\theta_\Omega\in \fv_\kappa(\Omega)$ 	and
	\[
	\fv_\kappa(\Lambda)+\sum_{z\in\grL} \theta_{\{z\}}\mc{M}_d
	\]
	is weak-$*$ dense in $\mc{M}_d$. It thus follows from Lemma \ref{L:TotalRange} that 
	\begin{equation}\label{Eq:range}
	\frk{H}=\bigvee\{\ran \phi(T): \phi\in \theta_{\{z\}} \M_d \text{ for some }z\in \Lambda\}=\bigvee_{z\in\grL} \ran\grj_{\{z\}}(T).
	\end{equation}
	Assume $z\in\grL$ is such that $ \theta_{\{z\}}(T)\neq 0$ and set  $R=T|_{\ran\theta_{\{z\}}(T)}$. Note that 
	\[
	\theta_{\{z\}} (1-\theta_{\{z\}})\in\frk{v}_k(\grL)\subset \Ann(T)
	\]
	and thus $\theta_{\{z\}}(T)$ is an idempotent. Likewise,
	\[
	\theta_{\{z\}}(x_j-z_j)^{\kappa+1}\in\frk{v}_k(\grL)\subset \Ann(T)
	\]
	for $1\leq j\leq d$.  It follows that 
	\[
	(T_j-z_j I)^{\kappa+1} \theta_{\{z\}}(T)=0, \quad 1\leq j\leq d
	\]
	so that
	\[
	((R_1-z_1I)^{\kappa+1},(R_2-z_2I)^{\kappa+1},\ldots, (R_d-z_dI)^{\kappa+1})
	\]
	is simply the zero $d$-tuple, and hence its Taylor spectrum is the origin in $\bC^d$. By the spectral mapping theorem \cite[Corollary IV.30.11]{muller2007}, we must have $\sigma(R)=\{z\}$. However, it follows from \cite[Lemma III.13.4]{VasilescuBook}  that $\sigma(R)\subset\sigma(T)$ so $z\in \sigma(T)$. This shows that $\theta_{\{z\}}\in \Ann (T)$ for every $z\in\grL\setminus\sigma(T)$. In particular, we note that 
	\[
	\theta_{\{z\}}\fv_\kappa(\sigma(T)\cap \Lambda)\subset \Ann (T)
	\]
	for all $z\in\grL$. For $\phi\in \fv_\kappa(\sigma(T)\cap \Lambda)$, we then have 
	\[
	\ran \theta_{\{z\}}(T)\subset \ker \phi(T), \quad z\in \Lambda
	\]
	and therefore $\phi(T)=0$ in light of Equation (\ref{Eq:range}). We conclude that 
	\[
	\fv_\kappa(\sigma(T)\cap \Lambda)\subset \Ann(T).
	\] 
	Since $\fv_\kappa(\Lambda)\subset \Ann(T)$, it follows from (i) and Theorem \ref{T:ISDim0} that
	\[
	\sigma(T)\cap \bB_d\subset \ZBd(\Ann (T))\subset \Z_{\bB_d}(\fv_\kappa(\Lambda))=\Lambda
	\]
	whence
	\[
	\sigma(T)\cap \bB_d=\sigma(T)\cap \Lambda.
	\]
	Therefore
	\[
	\fv_\kappa(\sigma(T)\cap \bB_d)\subset \Ann(T).
	\]
	Finally, let  $z\in \bB_d \setminus \sigma(T)$. As noted above, we have $\sigma(T)\cap \bB_d\subset \Lambda$ so that $\sigma(T)\cap \bB_d$ is an interpolating sequence, and thus so is $(\sigma(T)\cap \bB_d)\cup\{z\}$ by virtue of Corollary \ref{C:ISplus1}. Invoking Lemma \ref{L:GetIdems}, we find $\theta\in\frk{v}_k(\sigma(T)\cap\BB_d)$ such that $\theta(z)=1$. In particular, $\theta \in \Ann(T)$ so $z\in \bB_d\setminus \Z_{\bB_d}(\Ann(T))$. We conclude that
	\[
	 \Z_{\bB_d}(\Ann(T))\subset \sigma(T)\cap \bB_d
	\]
	so in fact equality holds.	
\end{proof}

A more thorough exploration of the relationship between the Taylor spectrum and annihilating ideals will be undertaken in an upcoming paper. For now, we turn to elucidating the structure of AC commuting row contractions whose Taylor spectrum is a singleton. As motivation, we first consider the univariate situation. Let $T\in B(\fH)$ be an AC contraction with non-trivial annihilating ideal and with $\sigma(T)=\{\lambda\}$ for some $\lambda\in \bB_1$. It then follows from \cite[Theorem 4.11]{bercovici1988} that  
$\Ann(T)$ is generated by some power of the Blaschke factor with root $\lambda$, and in particular $T-\lambda I$ is a nilpotent operator. As the next result shows, similar statements hold true for AC commuting row contractions under a topological assumption on the zero set of the annihilating ideal. 

We say that a commuting $d$-tuple $T=(T_1,\ldots,T_d)$ is \emph{nilpotent} if for each $1\leq j\leq d$ there is a positive integer $n_j$ such that $T_j^{n_j}=0$.

\begin{theorem}
	Let $T=(T_1,\dots,T_d)$ be an AC commuting row contraction and let $z\in \bB_d$ be an isolated point of $\Z_{\bB_d}(\Ann (T))$. Assume that $\sigma(T)=\{z\}$.	Then, $\PI(\Ann (T),z)$ is a weak-$*$ dense subset of $\Ann(T)$ and $T-zI$ is nilpotent.
	\label{T:GetNilp}
\end{theorem}
\begin{proof}
		Using the fact that $\mc{O}(z)$ is Noetherian, we can find $\psi_1,\dots,\psi_m\in\Ann(T)$ such that 
		\[
		\ip{[\psi]_z:\psi\in\Ann(T)}=\ip{[\psi_1]_z,\cdots,[\psi_m]_z}.
		\]
		On the other hand, it follows from Lemma \ref{L:GetPIs} that
		\begin{equation}\label{Eq:eqgerms}
		\ip{[\psi]_z:\psi\in\Ann(T)}=\ip{[p]_z:p\in \fp(\Ann(T),z)}.
		\end{equation}
		Let $p\in \fp(\Ann(T),z)$. There are functions $g_1,\cdots,g_m$ analytic on a neighborhood of $z$ such that
		\[ [p]_z=\sum_{j=1}^m [\psi_j]_z[g_j]_z.  \]
		In particular, there is a small open ball $B$ centred at $z$ on which the functions $g_1,\cdots,g_m$ are defined and holomorphic, and are such that
		$p=\sum_{j=1}^m\psi_j g_j$ everywhere on $B$. Applying the functional calculus to this equality and invoking Theorem \ref{T:EqualFC}, we find
		\[ p(T)=\sum_{j=1}^m \psi_j(T)g_j(T) = 0 \]
		since $\psi_1,\ldots,\psi_m\in \Ann(T)$. Thus, $p\in\Ann(T)$. We conclude that 
		\[
		\frk{p}(\Ann(T),z)\subset \Ann(T).
		\]
		Using Lemma \ref{L:GetPIs}, we find a positive integer $\kappa$ such that 
		\begin{equation}\label{Eq:inclann}
		\frk{m}_z^\kappa\subset  \fp(\Ann(T),z)\subset \Ann(T).
		\end{equation}
		In particular, $(x_j-z_j)^\kappa\in \Ann(T)$ for every $1\leq j\leq d$, whence the $d$-tuple $T-zI$ is nilpotent.

		It remains only to show that $\PI(\Ann (T),z)$ is weak-$*$ dense in $\Ann(T)$. Fix $\psi\in\Ann(T)$. Using Equation (\ref{Eq:eqgerms}), there are polynomials $q_1,\ldots,q_m\in \fp(\Ann(T),z)$ and functions $f_1,\ldots,f_m$ holomorphic on a neighbourhood of $z$ such that
\[
[\psi]_z=\sum_{j=1}^m [f_j]_z [q_j]_z.
\]
	For each $1\leq j\leq m,$ upon writing $f_j$ as a power series convergent around $z$, we see that there is another function $g_j$ holomorphic near $z$ such that 
	$
	[g_j]_z\in \widetilde{\fm_z}^{\kappa}
	$ 
	and a polynomial $r_j$ such that 
	\[
	[f_j]_z=[r_j]_z+[g_j]_z.
	\]
Set $p=\sum_{j=1}^m r_j q_j\in \fp(\Ann(T),z)$. Thus, there is $[g]_z\in \widetilde{\fm_z}^{\kappa}$ such that
\[
[\psi]_z=[p]_z+[g]_z.
\]
In particular, we infer that
	\[
	 \frac{\pd^\alpha }{\pd x^\alpha} (\psi-p)(z)=0
	\]
		for every $\alpha\in \bN^d$ such that $|\alpha|\leq \kappa-1$. By virtue of Theorem \ref{T:GleasonsTrick}, for each $\alpha\in \bN^d$ with $|\alpha|=\kappa $ there is a  multiplier $\phi_\alpha\in \M_d$ such that 
		\[
		\psi-p=\sum_{|\alpha|=\kappa} (x-z)^\alpha \phi_\alpha.
		\]
		In particular, this means that $\psi-p$ belongs to the weak-$*$ closure of $\fm_z^{\kappa}$ in $\M_d$. Invoking (\ref{Eq:inclann}), we see that 
		\[
		\psi=(\psi-p)+p
		\]
		belongs to the weak-$*$ closure of $\fp(\Ann(T),z)$. We conclude that  $\PI(\Ann (T),z)$ is weak-$*$ dense in $\Ann(T)$.
		\end{proof}

		We remark here that the zero set of the annihilating ideal of a single AC contraction is a Blaschke sequence, all the points of which are isolated.
		Thus, the topological assumption on the zero set in the previous result is automatically satisfied in one variable.
		
We will need to apply Theorem \ref{T:GetNilp} when $\sigma(T)$ is discrete but contains more than a single point. For this purpose,  we introduce the following procedure which allows us to isolate points in the spectrum.

\begin{lemma}
	Let $T=(T_1,\dots,T_d)$ be an AC commuting row contraction and suppose $z\in\sigma(T)\cap\BB_d$ is an isolated point of $\ZBd(\Ann (T))$.
	Let $B$ be an open ball around $z$ with $\cc{B}\subset \BB_d$ such that 
	\[
	\cc{B}\cap\ZBd(\Ann T)=\{z\}.
	\]
	If $\chi_B$ denotes the characteristic function of $B$, then $\ran \chi_B(T)$ is a non-zero subspace which coincides with
	\[
	\bigcap\{ \ker p(T):p\in \fp(\Ann(T),z)\}.
	\]
	\label{L:RanIdemEqKerPI}
\end{lemma}
\begin{proof}
	Note that $\cc{B}$ is disjoint from $\Z_{\bB_d}(\Ann(T))\setminus \{z\}$.
	In light of part (1) of Theorem \ref{T:spectrum}, we see that $\cc{B}$ is also disjoint from $\sigma(T)\setminus \{z\}$. In particular, $\chi_B$ is holomorphic on a neighbourhood of $\sigma(T)$, and $\chi_B(T)$ is a well-defined idempotent in $\{T_1,\ldots,T_d\}''$. Set $\frk{M}=\ran \chi_B(T)$.
	By Theorem \ref{T:Taylorprop}, we find $\frk{M}\neq \{0\}$ and $\sigma(T|_{\frk{M}})=\{z\}$. Furthermore, we note that
	\[
	\Ann(T)\subset \Ann(T|_{\fM})
	\]
	whence 
	\[
	\Z_{\bB_d}(\Ann(T|_\fM))\subset\ZBd(\Ann (T))
	\]
	and $z$ is an isolated point of $\Z_{\bB_d}(\Ann(T|_\fM))$. We may thus apply Theorem \ref{T:GetNilp} to conclude that $\Ann(T|_{\frk{M}})$ is generated by the ideal $\frk{p}(\Ann(T|_{\frk{M}}),z)$, so in particular
	\[
	\fp(\Ann(T),z)\subset \frk{p}(\Ann(T|_{\frk{M}}),z)\subset  \Ann(T|_\fM).
	\]
	Let now $p\in \fp(\Ann(T),z)$. Then, $p\in \Ann(T|_\fM)$ so
	\[
	0=p(T|_\fM)=p(T)|_{\fM}
	\]
	 and hence $\frk{M}\subset\ker  p(T)$. This shows that 
	\[
	\fM\subset \bigcap\{ \ker p(T):p\in \fp(\Ann(T),z)\}.
	\]
	To show the reverse inclusion, we put
	\[
	\fK= \bigcap\{ \ker p(T):p\in \fp(\Ann(T),z)\}
	\]
	and let $R=T|_\fK$. It follows from Lemma \ref{L:GetPIs} that there is a positive integer $\kappa$ such that 
	\[
	\fm_z^\kappa\subset \fp(\Ann(T),z)\subset \Ann(R).
	\]
	In particular, we see that 
	\[
	((R_1-z_1I)^\kappa,(R_2-z_2 I)^\kappa,\ldots,(R_d-z_dI)^\kappa)
	\]
	is simply the zero $d$-tuple, and hence its Taylor spectrum is the origin in $\bC^d$. By the spectral mapping theorem \cite[Corollary IV.30.11]{muller2007}, we must have $\sigma(R)=\{z\}$.
	But then $\chi_B$ is identically $1$ on a neighbourhood of $\sigma(R)$, and it follows that $\chi_B(R)=I$.
	Let 
	$
	X:\fK\to\frk{H}
	$
	be the inclusion map, and note that $XR_j=T_jX$ for $1\leq j\leq d$.
	It follows from Theorem \ref{T:Taylorprop} that 
	\[
	X=X\chi_B(R)=\chi_B(T)X.
	\]
	Thus,
	\[
	\fK=\ran X\subset \ran \chi_B(T)=\fM
	\]
	and the proof is complete.
\end{proof}

Theorem \ref{T:GetNilp} and Lemma \ref{L:RanIdemEqKerPI} taken together hint at a possible approach to construct Jordan-type decompositions. However, to deal with infinite spectra  this procedure would need to be applied inductively infinitely many times, thus causing significant problems regarding convergence for instance. Whenever the zero set of the annihilating ideal forms an interpolating sequence, these difficulties can be circumvented as the next developments showcase.

\begin{lemma}
	Let $T=(T_1,\dots,T_d)$ be an AC commuting row contraction on some Hilbert space $\frk{H}$. Let $\Lambda\subset \bB_d$ be an interpolating sequence such that $\Lambda=\Z_{\bB_d}(\Ann(T))$. Assume that there is a non-negative integer $\kappa$ such that $\fv_\kappa(\Lambda)\subset \Ann(T)$. For each $z\in \Lambda$, let
	\[
	\fK_z=\bigcap\{ \ker p(T):p\in \fp(\Ann(T),z)\}.
	\]
	Then, the following statements hold.
	\begin{enumerate}

	\item[\rm{(i)}] Let $z\in \Lambda$ and let $\theta\in \fv_\kappa(\Lambda\setminus\{z\})$ such that $1-\theta\in \fv_\kappa(\{z\})$. Then, the subspace $\fK_z$ is non-zero and coincides with $\ran\theta(T)$.
		
	\item[\rm{(ii)}] We have that
	\[
		\frk{H}=\bigvee_{z\in \Lambda} \fK_z.
	\]
	\item[\rm{(iii)}] Let $z\in \Lambda$. Then, the ideal $\Ann(T|_{\fK_z})$ is the weak-$*$ closure of $\fp(\Ann(T),z)$ in $\M_d$.
	\end{enumerate}
	\label{L:GleasonCondLemma}
\end{lemma}
\begin{proof}
(i) Fix $z\in \Lambda$. Since $\Lambda$ is an interpolating sequence, $z$ is an isolated point of $\Lambda$, and hence of $\Z_{\bB_d}(\Ann(T))$ by assumption. Furthermore, we see that $z\in \sigma(T)\cap \bB_d$ by Theorem \ref{T:spectrum}. The fact that $\fK_z$ is non-zero then follows immediately from Lemma \ref{L:RanIdemEqKerPI}. Next, let $\theta\in \fv_\kappa(\Lambda\setminus\{z\})$ such that $1-\theta\in \fv_\kappa(\{z\})$. Let $p\in \fp(\Ann(T),z)$, which means that
\[
[p]_z\in \ip{[\phi]_z:\phi\in \Ann(T)}.
\]
Hence, there are multipliers $\phi_1,\ldots,\phi_m\in \Ann(T)$ and functions $f_1,\ldots,f_m$ holomorphic on a neighbourhood of $z$ such that
\[
[p]_z=\sum_{j=1}^m [f_j]_z [\phi_j]_z.
\]
	For each $1\leq j\leq m$, upon writing $f_j$ as a power series convergent around $z$, we see that there is another function $g_j$ holomorphic near $z$ such that 
	$
	[g_j]_z\in \widetilde{\fm_z}^{\kappa+1}
	$ 
	and a polynomial $r_j$ such that 
	\[
	[f_j]_z=[r_j]_z+[g_j]_z.
	\]
Set $\phi=\sum_{j=1}^m r_j \phi_j\in \Ann(T)$. Thus, there is $[g]_z\in \widetilde{\fm_z}^{\kappa+1}$ such that
\[
[p]_z=[\phi]_z+[g]_z.
\]
In particular, we infer that
$
p-\phi\in \fv_{\kappa}(\{z\})
$
whence
\[
(p-\phi)\theta\in \fv_\kappa(\Lambda)\subset \Ann(T)
\]
and therefore $p\theta\in \Ann(T)$. Consequently, we find
$
p(T)\theta(T)=0
$
so that $\ran\theta(T)\subset \ker p(T)$. This shows that $\ran \theta(T)\subset \fK_z$. Conversely, we note that by Lemma \ref{L:FinVanOrd} we have $\fm_z^{\kappa+1}\subset \fp(\Ann(T),z)$ so that
\[
\fK_z\subset  \bigcap_{|\alpha|=\kappa+1} \ker(T-z I)^\alpha.
\]
Now, we have $1-\theta\in \fv_{\kappa}(\{z\})$, so that by Theorem \ref{T:GleasonsTrick} for each $\alpha\in \bN^d$ with $|\alpha|=\kappa+1$ there is $\psi_\alpha\in\M_d$ such that
\[
1-\theta=\sum_{|\alpha|=\kappa+1}(x-z)^\alpha \psi_\alpha.
\]
We conclude that
\[
 \fK_z\subset \bigcap_{|\alpha|=\kappa+1} \ker(T-z I)^\alpha\subset \ker (I-\theta(T)).
\]
Now, 
\[
\theta^2-\theta=\theta(1-\theta)\in \fv_\kappa(\Lambda)\subset \Ann(T)
\]
so that $\theta(T)$ is idempotent and $ \ker (I-\theta(T))=\ran \theta(T)$. We conclude that $\fK_z\subset \ran \theta(T)$, and statement (i) is established.

(ii) Apply Lemma \ref{L:GetIdems} to find for every $z\in \Lambda$ a multiplier $\theta_z\in \fv_\kappa(\Lambda\setminus \{z\})$ such that $1-\theta_z\in \fv_\kappa(\{z\})$ and with the property that the ideal
	\[
	\fv_\kappa(\Lambda)+\sum_{z\in\grL} \theta_{z}\mc{M}_d
	\]
	is weak-$*$ dense in $\mc{M}_d$. By Lemma \ref{L:TotalRange}, we infer that
	\begin{align*}
	\frk{H}&=\bigvee\{\ran \phi(T): \phi\in \theta_z \M_d \text{ for some }z\in \Lambda\}\\
	&=\bigvee_{z\in\grL} \ran\grj_{z}(T)=\bigvee_{z\in\grL} \fK_z
	\end{align*}
	where the last equality follows from statement (i). Thus, statement (ii) holds.

(iii) Let $z\in \Lambda$. It follows immediately from the definition of $\fK_z$ that
	\[
	\fp(\Ann(T),z)\subset \Ann(T|_{\fK_z}).
	\]
	Let now $\phi\in\Ann(T|_{\fK_z})$.
	By Theorem \ref{T:GleasonsTrick} there is a polynomial $p$ and a multiplier $\psi$ in the ideal generated by $\frk{m}_z^{\kappa+1}$ such that 
	\[ \phi=p+\psi. \]
	As noted in the proof of (i), we have  $\frk{m}_z^{\kappa+1}\subset \fp(\Ann(T),z)$, so in fact $\psi$ belongs to the weak-$*$ closure of $ \fp(\Ann(T),z)$ in $\M_d$.
	Since $T$ is AC, we can then infer that
	\[
	\fK_z\subset \ker \psi(T).
	\]
	Applying statement (i) to the function $\theta_z\in \M_d$ defined in the proof of (ii) above, we find that
	\[
	\ran \theta_z(T)=\fK_z
	\]
	whence
	\[
	\phi(T)\theta_z(T)=\phi(T|_{\fK_z})\theta_z(T)=0
	\]
	and
	\begin{align*}
	0&=\phi(T)\theta_z(T)=p(T)\theta_z(T)+\psi(T)\theta_z(T)\\
	&=p(T)\theta_z(T). 
	\end{align*}
	That is, $p\theta_z\in\Ann(T)$ and therefore 
	\[
	[p]_z[\theta_z]_z\in \ip{[\tau]_z:\tau\in\Ann(T)}.
	\]
	Because $\theta_z(z)=1$, it follows that $[\theta_z]_z$ is invertible in $\O(z)$, hence 
	\[
	[p]_z\in \ip{[\tau]_z:\tau\in\Ann(T)}
	\]
	 and therefore $p\in\fp(\Ann(T),z)$. Hence,
	 $
	 \phi=p+\psi
	 $
	 belongs to the weak-$*$ closure of $\fp(\Ann(T),z)$  in $\M_d$ and statement (iii) follows.
\end{proof}

Finally, we arrive at the main result of this section, which is also the central result of the paper. Therein, we obtain a Jordan-type decomposition for AC commuting row contractions whose annihilating ideal is a higher order vanishing ideal of some interpolating sequence.

\begin{theorem}
	Let $T=(T_1,\dots,T_d)$ be an AC commuting row contraction.
	Let $\Lambda\subset \bB_d$ be an interpolating sequence and let $\kappa$ be a non-negative integer such that $\fv_\kappa(\Lambda)\subset \Ann(T)$.
	Then, for each $z\in \Z_{\bB_d}(\Ann(T))$ there is a commuting nilpotent $d$-tuple $N^{(z)}$ such that $zI+N^{(z)}$ is an AC commuting row contraction whose annihilating ideal is generated by $\fp(\Ann (T),z)$.
	Furthermore, $T$ is similar to 
	\[
	\bigoplus_{z\in\ZBd(\Ann T)}(zI+N^{(z)}).
	\]
	\label{T:Decomposition}
\end{theorem}
\begin{proof}
Let $\Lambda_0=\Z_{\bB_d}(\Ann(T))$. We see that 
\[
\Lambda_0\subset \Z_{\bB_d}(\fv_\kappa(\Lambda))=\Lambda
\]
where the last equality follows from Theorem \ref{T:ISDim0}. In particular, $\Lambda_0$ is also an interpolating sequence. Now, Theorem \ref{T:spectrum} implies that $\Lambda_0=\sigma(T)\cap \bB_d$ and $\fv_{\kappa}(\Lambda_0)\subset \Ann(T)$, so upon replacing $\Lambda$ by $\Lambda_0$ if necessary, it is no loss of generality to assume that 
\[
\Lambda=\Z_{\bB_d}(\Ann(T))=\sigma(T)\cap \bB_d.
\]
By Lemma \ref{L:GetIdems}, for each subset $\Omega\subset \Lambda$ there is a multiplier $\theta_\Omega\in \fv_\kappa(\Lambda\setminus \Omega)$ such that $1-\theta_\Omega\in \fv_\kappa(\Omega)$ and
	\[
	\sup_{\Omega\subset \Lambda}\|\theta_\Omega\|<\infty.
	\]
In particular, we note that
\begin{equation}\label{Eq:derivative0}
\frac{\pd^\alpha \theta_\Omega}{\pd x^\alpha}(z)=0, \quad z\in \Lambda
\end{equation}
for every $\alpha\in \bN^d$ such that $ |\alpha|\geq 1$ and $|\alpha|\leq \kappa$. Furthermore, for each $\Omega\subset \Lambda$ we have that 
\[
\theta_\Omega(1-\theta_\Omega)\in \fv_{\kappa}(\Lambda)\subset \Ann(T)
\]
which implies that $\theta_\Omega(T)$ is idempotent. Now, given two subsets $\Omega_1,\Omega_2\subset \Lambda$ we denote their symmetric difference by
\[
\Omega_1 \triangle \Omega_2=(\Omega_1\setminus \Omega_2)\cup (\Omega_2\setminus \Omega_1).
\]
A routine verification reveals that the function
\[
\grj_{\grW_1}+\grj_{\grW_2}-2\grj_{\grW_1}\grj_{\grW_2}-\grj_{\grW_1\triangle\grW_2}
\]
vanishes everywhere on $\Lambda$. Combining this observation with Equation (\ref{Eq:derivative0}), it is readily seen that
\[
\grj_{\grW_1}+\grj_{\grW_2}-2\grj_{\grW_1}\grj_{\grW_2}-\grj_{\grW_1\triangle\grW_2}\in \fv_{\kappa}(\Lambda)\subset \Ann(T)
\]
for every subsets $\Omega_1,\Omega_2\subset \Lambda$.
Thus
	\[ \grj_{\grW_1}(T)+\grj_{\grW_2}(T)-2\grj_{\grW_1}(T)\grj_{\grW_2}(T)=\grj_{\grW_1\grD\grW_2}(T). \]
There exists an invertible operator $Y$ such that $\{Y\grj_{\{z\}}(T)Y^{-1}\}_{z\in \Lambda}$ is a family of pairwise orthogonal self-adjoint projections (see \cite[Lemma 1.0.3]{gifford1997} or \cite[Exercise 9.13]{paulsen2002} for instance).

For each $z\in \Lambda$, we put
\[
\fK_z=\bigcap\{ \ker p(T):p\in \fp(\Ann(T),z)\}.
\]
Applying Lemma \ref{L:GleasonCondLemma}, we see that 
$
\frk{H}=\bigvee_{z\in\grL}\fK_z.
$
Moreover, for every $z\in \Lambda$ we have that $\fK_z=\ran \theta_{\{z\}}(T)$ is a non-zero invariant subspace for $T$ and that $\Ann(T|_{\fK_z})$ is the weak-$*$ closure of $\fp(\Ann(T),z)$.
Thus $Y\fK_z$ is orthogonal to $Y\fK_w$ whenever $z,w\in \Lambda$ are distinct, so that
\[
\fH=\bigoplus_{z\in \Lambda}Y\fK_z
\]
where $\fH$ denotes the Hilbert space on which $T$ acts.
Since every $\fK_z$ is invariant for $T$, we have
\[
YTY^{-1}=\bigoplus_{z\in \Lambda} Y_zT|_{\fK_z}Y_z^{-1}
\]
where
\[
Y_z=Y|_{\fK_z}:\fK_z\to Y\fK_z.
\]
Let $X:\frk{H}\to \bigoplus_{z\in\Lambda}\fK_z$ be given by
\[ Xh = (Y^{-1}P_{Y\frk{K}_z}Yh)_{z\in\Lambda} \]
for $h\in\frk{H}$. 
Then $X$ is a boundedly invertible linear map with the property that
\[ XTX^{-1} = \bigoplus_{z\in\Lambda} T|_{\fK_z}. \]

Fix a point $z$ of $\Lambda=\sigma(T)\cap \bB_d=\Z_{\bB_d}(\Ann(T))$, which is necessarily isolated. Invoke Lemma \ref{L:RanIdemEqKerPI}  to find an open ball $B$ containing $z$ such that $\ol{B}\subset \bB_d$ and $\sigma(T)\cap\cc{B}=\{z\}$, and if $\chi_B$ denotes the characteristic function of $B$ then 
\[
\ran\chi_B(T)=\frk{K}_z.
\]
In turn, use Theorem \ref{T:Taylorprop} to find
$
\sigma(T|_{\fK_z})=\{z\}.
$
Since
$
\Ann(T)\subset \Ann(T|_{\fK_z}),
$
we conclude that
\[
\Z_{\bB_d}(\Ann(T|_{\fK_z}))\subset \Z_{\bB_d}(\Ann(T))
\]
so that $z$ is an isolated point of $\Z_{\bB_d}(\Ann(T|_{\fK_z}))$. We may apply Theorem \ref{T:GetNilp} to see that 
\[
T|_{\frk{K}_z}=zI+N^{(z)}
\]
for some nilpotent $d$-tuple $N^{(z)}$ acting on $\fK_z$ whose annihilating ideal is generated by $\fp(\Ann (T),z)$ (by Lemma \ref{L:GleasonCondLemma}).

\end{proof}

We mention that the previous theorem generalizes \cite[Corollary 3.3]{clouatreSIM1} in two ways: it extends it to the multivariate setting, and it allows for a wider range of annihilating ideals. The reader may also wish to compare with \cite[Theorem 5.7]{clouatreSIM1}. We close this section with a reformulation of Theorem \ref{T:Decomposition} in the case where $\kappa=0$.

\begin{corollary}\label{C:decompkappa0}
Let $T=(T_1,\dots,T_d)$ be an AC commuting row contraction.
Let $\Lambda\subset \bB_d$ be an interpolating sequence such that $\fv_0(\Lambda)\subset \Ann(T)$.
Then, $T$ is similar to 
	$
	\bigoplus_{z\in\ZBd(\Ann T)}zI.
	$
If in addition we assume that $T$ has a cyclic vector, then $T$ is similar to 
$
	\bigoplus_{z\in\ZBd(\Ann T)} z.
	$
\end{corollary}
\begin{proof}
 By virtue of Theorem \ref{T:Decomposition}, for each $z\in \Z_{\bB_d}(\Ann(T))$ there is a commuting nilpotent $d$-tuple $N^{(z)}$ such that $z I+N^{(z)}$ is an AC commuting row contraction whose annihilating ideal is generated by $\fp(\Ann(T),z)$.
 Furthermore, $T$ is similar to 
	\[
	\bigoplus_{z\in\ZBd(\Ann T)}(zI+N^{(z)}).
	\]
It is readily verified that $\frk{m}_z=\PI(\fv_0(\grL),z)$, whence 
\[
\fm_z\subset \PI(\Ann(T),z)\subset \Ann(z I+N^{(z)})
\]
and so $N^{(z)}=0$.
 Hence $T$ is in fact similar to $\bigoplus_{z\in\ZBd(\Ann T)}zI$.
 Finally, if $T$ has a cyclic vector, then so does $\bigoplus_{z\in\ZBd(\Ann T)}zI$, which forces the identity operators appearing in this decomposition to act on one-dimensional spaces, whence $T$ is indeed similar to $	\bigoplus_{z\in\ZBd(\Ann T)}z$.
\end{proof}

\section{Application: an operator theoretic characterization of interpolating sequences}\label{S:siminterp}

As a first application of Theorem \ref{T:Decomposition}, in this section we explore a characterization of interpolating sequences phrased purely in operator theoretic terms. More precisely, we seek to obtain a multivariate version of \cite[Theorem 4.4]{clouatreSIM1}. 

We begin by recording a simple observation.

\begin{lemma}\label{L:gen}
	Let $\{\omega_n:n\in \bN\}\subset \M_d$ be a collection of multipliers, let $\fH$ be separable Hilbert space and let
	\[
		\Omega:\bB_d\to B(\fH,\bC)
	\]
	be the associated row multiplier defined as
	\[
	\Omega(z)=\begin{bmatrix} \omega_1(z) & \omega_2(z) & \ldots & \end{bmatrix}
	\]
	for every $z\in \bB_d$. Assume that $\Omega$ is bounded. Let $\fa\subset \M_d$ be a weak-$*$ closed ideal such that 
	\[
	\ol{\ran M_\Omega}=[\fa H^2_d].
	\]
	 Then, $\fa$ is the weak-$*$ closure of $\ip{\omega_n:n\in \bN}.$
\end{lemma}
\begin{proof}
Let $\fb\subset\M_d$ denote the weak-$*$ closure of $\ip{\omega_n:n\in \bN}.$ Then
	\[
		[\frb H^2_d]=\ol{\sum_n \omega_n H^2_d}=\ol{\Omega H^2_d(\fH)}=[\fa H^2_d]
	\]
so that $\frb=\fa$ by Theorem \ref{T:Beurling}.
\end{proof}

Another elementary fact we single out relates to compressions of partial isometries.

\begin{lemma}\label{L:partialisom}
The following statements hold.
\begin{enumerate}

\item[\rm{(i)}] Let $\{V_n:n\in \bN\}$ be a family of contractions on some Hilbert space. Assume that the row operator 
	\[
		V=\begin{bmatrix} V_1 & V_2 & \ldots & \end{bmatrix}	
	\]
	is a partial isometry.
 	Let $\fM$ be a closed subspace which is coinvariant for $V_n$ for each $n\in \bN$ and such that $\fM^\perp\subset \ran V$.
  Then, $ P_{\fM}VP_{\fM}$ is a partial isometry.
  
  \item[\rm{(ii)}] Let $\fa\subset \M_d$ be a weak-$*$ closed ideal. Let $\fH$ be a Hilbert space and let $\Omega:\bB_d\to B(\fH,\bC)$ be an inner multiplier such that $[\fa H^2_d]\subset \ran M_\Omega$. Then,  $P_{\H_\fa} M_\Omega|_{\H_\fa}$ is a partial isometry.
\end{enumerate}
\end{lemma}
\begin{proof}
(i) Since $\fM^\perp\subset \ran V$, we see that 
\begin{align*}
P_{\fM}P_{\ran V}&=(I-P_{\fM^\perp})P_{\ran V}=P_{\ran V}-P_{\fM^\perp}\\
&=P_{\ran V\ominus \fM^\perp}=P_{\ran V\cap \fM}.
\end{align*}
 Using the fact that $\fM$ is coinvariant for each $V_n$, we find
\begin{align*}
( P_{\fM}VP_{\fM})( P_{\fM}VP_{\fM})^*&=\sum_{n=1}^\infty P_{\fM} V_n P_{\fM} V_n^* P_{\fM}=P_{\fM}\left( \sum_{n=1}^\infty V_n V_n^*\right)P_{\fM}\\
&=P_{\fM}VV^*P_{\fM}=P_{\fM}P_{\ran V}P_{\fM}\\
&=P_{\ran V\cap \fM}.
\end{align*}
We conclude that $ P_{\fM}VP_{\fM}$ is a partial isometry.

(ii) This follows immediately from (i).
\end{proof}

We remark that statement (ii) in the previous result is analogous to a classical fact \cite[Problem III.1.11]{bercovici1988}, which says that if $\theta\in H^\infty$ is an inner function and $\omega$ is an inner divisor of $\theta$, then $\omega(S_\theta)$ is a partial isometry
 Here, $S_\theta$ denotes the one-variable model operator.

Next, we obtain a sort of converse to Theorem \ref{T:Decomposition}. Roughly speaking, it says that a sequence can be determined to be strongly separated (see Subsection \ref{SS:DA}) if there exists a certain Jordan-type decomposition.

\begin{theorem}\label{T:simss}
	Let $\Lambda\subset \bB_d$ be a countable subset and let $\fa\subset \M_d$ be a weak-$*$ closed ideal. Assume that there is an invertible operator $X$ such that
	\[
	X^{-1}Z^\fa X=\bigoplus_{\lambda\in \Lambda} \lambda.
	\]
	Then, $\Lambda$ is a strongly separated sequence. Furthermore,  if we let $\fa_{\lambda}=\fv_0(\Lambda\setminus \{\lambda\})$ for each $\lambda\in \Lambda$, then there is an inner multiplier $\Omega_\lambda$ with the property that 
	\[
	\ran M_{\Omega_\lambda}=[\fa_\lambda H^2_d]
	\]
	and
	\[
		\inf_{\lambda\in \Lambda}\|\Omega_\lambda (\lambda)\|\geq \frac{1}{\|X\| \|X^{-1}\|}.
	\]

	\end{theorem}
\begin{proof}
	We identify $\bigoplus_{\lambda\in\Lambda}\CC$ with $\ell^2(\Lambda)$, and denote by $\{e_\lambda:\lambda\in \Lambda\}$ the standard orthonormal basis of $\ell^2(\Lambda)$.
	Let $D=\bigoplus_{\lambda\in\Lambda}\lambda$ which is an AC commuting row contraction on $\ell^2(\Lambda)$.
	By assumption, there exists an invertible operator $X:\ell^2(\Lambda)\to \H_{\fa}$ such that $XDX^{-1}=Z^\fa$.
	In particular, 
	\[ \fv_0(\Lambda)=\Ann(D)=\Ann(Z^\fa)=\frk{a} \]
  where the last equality follows from \cite[Lemma 2.10]{CT2018}.
	Given $\lambda\in\Lambda$, let $P_\lambda$ denote the projection onto $\CC e_\lambda$, and let $Q_\lambda = XP_\lambda X^{-1}$.
	Then $Q_\lambda$ commutes with $Z^\fa$, and by the Commutant Lifting Theorem \cite[Theorem 5.1]{BTV2001} there exists $\phi_\lambda\in\M_d$ such that $\phi_\lambda(Z^\fa)=Q_\lambda$ and
	\[ \|\phi_\lambda\|=\|Q_\lambda\| \leq \|X\|\|X^{-1}\|. \]
	Fix $\lambda\in \Lambda$. Plainly $DP_\lambda = \lambda P_\lambda$, whence $Z^\fa Q_\lambda = \lambda Q_\lambda$, and
	\[ Q_\lambda = Q_\lambda^2 = \phi_\lambda(Z^\fa)Q_\lambda = \phi_\lambda(\lambda)Q_\lambda, \]
	so that $\phi_\lambda(\lambda)=1$.
	When $\mu\in\Lambda\bksl\{\lambda\}$, we have
	\[ 0=Q_\lambda Q_\mu = \phi_\lambda(Z^\fa)Q_\mu = \phi_\lambda(\mu)Q_\mu, \]
	and thus $\phi_\lambda(\mu)=0$.
	We conclude that $\phi_\lambda \in \fa_\lambda$, and that the collection of multipliers $\{\phi_\lambda:\lambda\in\Lambda\}$ witnesses the fact that $\Lambda$ is strongly separated.

	Next, by Theorem \ref{T:Beurling}  there is a Hilbert space $\frk{X}$ such that for each $\lambda\in\Lambda$ there is an inner multiplier $\Omega_\lambda:\bB_d\to B(\frk{X},\bC)$ with $\ran \Omega_\lambda=[\fa_{\lambda} H^2_d]$.
	Because
	\[ \frk{a}=\fv_0(\Lambda)\subset \fa_\lambda, \]
	we have that $[\fa H^2_d]\subset \ran\Omega_\lambda$.
	In light of Lemma \ref{L:partialisom}, we infer that the row operator
	\[ \Omega_\lambda(Z^\fa)=P_{\H_\fa}M_{\Omega_\lambda}|_{\H_\fa}:\H_\fa\otimes\frk{X}\to\H_\fa \]
	is a partial isometry for every $\lambda\in\Lambda$. Consider now the row operator 
	\[
	\Omega_\lambda(D): \ell^2(\Lambda)\otimes\frk{X}\to \ell^2(\Lambda)
	\]
	defined as
	\[ \Omega_\lambda(D)v = (\Omega_\lambda(\mu)v_\mu)_{\mu\in\Lambda} \]
	for every $v=(v_\mu)_{\mu\in \Lambda}\in \ell^2(\Lambda)\otimes \fX$. We observe that $X\Omega_\lambda(D)X^{-1}=\Omega_\lambda(Z^\fa)$.

Fix $\lambda\in \Lambda$. Let $h\in\ran \Omega_\lambda(D)$. It is then readily verified that $Xh$ lies in the range $\Omega_\lambda(Z^\fa)$. Hence, we may choose $f\in\H_\fa\otimes\frk{X}$ such that $\Omega_\lambda(Z^\fa)f=Xh$ and $\|f\|=\|Xh\|$, which implies that
	\[
		\Omega_\lambda (D)X^{-1}f=X^{-1}\Omega_\lambda(Z^\fa)f=h.
	\]
	Note also that 
	\[
		\|X^{-1}f\|\leq \|X^{-1}\| \|f\|\leq \|X^{-1}\| \|X\| \|h\|.
	\]
 Let $\grw_1,\grw_2,\ldots$ be contractive multipliers such that
	\[ \grW_\lambda(z) = \begin{bmatrix} \grw_1(z) & \omega_2(z) & \ldots \end{bmatrix}, \quad z\in \bB_d. \]
	By Lemma \ref{L:gen}, $\fa_{\lambda}$ is the weak-$*$ closed ideal generated by $\{\omega_m:m\in\bN\}$. Since $\phi_\lambda\in \fa_{\lambda}$ and since $D$ is absolutely continuous, we conclude that $\phi_\lambda(D)$ lies in the weak-$*$ closure of the ideal in $B(\fH)$ generated by $\{\grw_m(D):m\in \bN\}$.  Since $e_\lambda=\phi_\lambda(D)e_\lambda$, it follows that 
	\[
	e_\lambda\in \bigvee_{m=1}^\infty \ran\grw_m(D)\subset  \ran \grW_\lambda(D).
	\]
	As see above, there is a $v=(v_\mu)_{\mu\in \Lambda}\in \ell^2(\Lambda)\otimes \X$ 
	with $\|v\|\leq \|X\| \|X^{-1}\|$ such that $\grW_\lambda(D)v=e_\lambda$. We conclude that
	\begin{align*}
	1&=\|e_\lambda\|^2=|\langle \Omega_\lambda(D) v,e_\lambda \rangle|=|\langle \Omega_\lambda(\lambda) v_\lambda,e_\lambda \rangle|\\
	&\leq \|\Omega_\lambda(\lambda)\| \|v\|\leq \|\Omega_\lambda(\lambda)\| \|X\| \|X^{-1}\|
	\end{align*}
	and thus
	\[
	\|\Omega_\lambda(\lambda)\|\geq  \frac{1}{ \|X\| \|X^{-1}\|}.
	\]
	
\end{proof}

One consequence of the previous theorem is that the sequence $\Lambda$ is both strongly separated and strongly separated by inner multipliers (see Subsection \ref{SS:DA}). This is no coincidence; these notions actually coincide. The proof of this fact requires the following technical tool.

\begin{theorem}\label{T:SSS}
	Let $\lambda\in \bB_d$ and let $\fa\subset \M_d$ be a weak-$*$ closed ideal.
	 Let $\delta>0$.
	 The following statements are equivalent.
	\begin{enumerate}
		\item[\rm{(i)}] There is $\omega\in \fa$ with $\|\omega\|\leq 1$ and  such that $|\omega(\lambda)|>\delta$.
		\item[\rm{(ii)}] There is a separable Hilbert space $\fH$ and an inner multiplier $\Omega:\bB_d\to B(\fH,\bC) $  such that $\|\Omega(\lambda)\|> \delta$ and $\ran \Omega=[\fa H^2_d]$.
	\end{enumerate}
\end{theorem}
\begin{proof}
	Assume first that there is $\omega\in \fa$ with $\|\omega\|\leq 1$ such that $|\omega(\lambda)|>\delta$. By Theorem \ref{T:Beurling}, there is a separable Hilbert space $\fH$ and an inner multiplier  $\Omega:\bB_d\to B(\fH,\bC)$ such that $\ran M_\Omega=[\fa H^2_d]$. Note then that
	\[
		\ol{\ran M_\omega}\subset [\fa H^2_d]=\ran \Omega.
	\]
	Using that $\ran M_\Omega=(\ker M_\Omega^*)^\perp$ and that $M_\omega M_\omega^*\leq I$, we see that
	\[
		M_\omega M_\omega^* = P_{(\ker M_\Omega^*)^\perp} M_\omega M_\omega^* P_{(\ker M_\Omega^*)^\perp} \leq P_{\ran M_\Omega}=M_\Omega M_\Omega^*. \]
	By \cite[Theorem 1.10]{MT2012}, we conclude that there is a contractive multiplier $\Theta: \bB_d\to B(\bC,\fH)$ such that $\omega=\Omega\Theta$.
 In particular, we see that
	\[
		\delta< |\omega(\lambda)|\leq \|\Omega(\lambda)\| \|\Theta(\lambda)\|\leq \|\Omega(\lambda)\|.
	\]

	Conversely, assume that there is a separable Hilbert space $\fH$ and an inner multiplier $\Omega:\bB_d\to B(\fH,\bC)$  with $\ran \Omega=[\fa H^2_d]$ such that $\|\Omega(\lambda)\|> \delta$. There are contractive multipliers $\{\omega_n:n\in\bN\}$ such that 
	\[
		\Omega(z)=\begin{bmatrix} \omega_1(z) &\omega_2(z) & \ldots &  \end{bmatrix}, \quad z\in \bB_d.
	\]
	Then, by Lemma \ref{L:gen}, we see that $\omega_n\in \fa$ for every $n$.
	Moreover, we observe that
	\[
		\sum_{n=1}^\infty|\omega_n(\lambda)|^2=\|\Omega(\lambda)\|^2> \delta^2.
	\]
	Choose $N\geq 1$ large enough so that 
	\[
		\sum_{n=1}^N|\omega_n(\lambda)|^2>\delta^2.
	\]
	Next, choose $c_1,c_2,\ldots,c_N\in \bC$ such that $\sum_{n=1}^N |c_n|^2=1$ and
	\[
		\sum_{n=1}^N c_n \omega_n(\lambda)=\left(\sum_{n=1}^N|\omega_n(\lambda)|^2 \right)^{1/2}.
	\]
	Set $\omega=\sum_{n=1}^N c_n \omega_n\in \fa$.
	Since
	\[
		\omega=\begin{bmatrix} c_1 & c_2 & \ldots & c_N \end{bmatrix}\begin{bmatrix} \omega_1\\ \omega_2 \\ \vdots \\ \omega_N \end{bmatrix}
	\]
	we see that $\|M_\omega\|\leq 1$.
	Finally, we find
	\[
		\omega(\lambda)=\sum_{n=1}^N c_n \omega_n(\lambda)=\left(\sum_{n=1}^N|\omega_n(\lambda)|^2 \right)^{1/2}>\delta. \qedhere
	\]
\end{proof}

We can now show that the notions of strong separation and of strong separation by inner multipliers coincide.

\begin{corollary}\label{C:SSS}
	Let $\Lambda=\{\lambda_n:n\in \bN\}\subset \bB_d$ be a sequence and let $\delta>0$.
  Then, the following statements are equivalent.
	\begin{enumerate}
		\item[\rm{(i)}] For every $n\in \bN$, there is a contractive multiplier $\omega_n\in \M_d$ with $|\omega_n(\lambda_n)|>\delta$
 and such that $\omega_n(\lambda_m)=0$ for every $m\neq n$.
    \item[\rm{(ii)}] For every $n\in \bN$, there is a separable Hilbert space $\fH_n$ and an inner multiplier $\Omega_n:\bB_d\to B(\fH_n,\bC)$ with $\|\Omega_n(\lambda_n)\|>\delta$ and such that $\Omega_n(\lambda_m)=0$ for every $m\neq n$.
	\end{enumerate}
\end{corollary}
\begin{proof}
	For every $n\in \bN$, let $\fa_n=\fv_0(\Lambda\setminus \{\lambda_n\})$.
	Assume that (i) holds and fix $n\in \bN$.  Then, we see that $\omega_n\in \fa_n$, so by Theorem \ref{T:SSS} there is a separable Hilbert space $\fH_n$ and an inner multiplier $\Omega_n:\bB_d\to B(\fH_n,\bC) $ such that $\ran M_{\Omega_n}=[\fa_n H^2_d]$ and $\|\Omega_n(\lambda_n)\|>\delta$.
	Now, there are contractive multipliers $\{\theta_k:k\geq 1\}$ such that
	\[
		\Omega_n(z)=\begin{bmatrix} \theta_1(z) & \theta_2(z) & \ldots &  \end{bmatrix}, \quad z\in \bB_d.
	\]
	Lemma \ref{L:gen} implies that $\theta_k\in \fa_n$ for every $k\in \bN$.
	 In particular, for every $k\in \bN$ and every $m\neq n$  we have $\theta_k(\lambda_m)=0$.
	 Thus, $\Omega_n(\lambda_m)=0$ if $m\neq n$. We conclude that (ii) holds.

	Conversely, assume that (ii) holds and fix $n\in \bN$.  There are contractive multipliers $\{\theta_k:k\geq 1\}$ such that
	\[
		\Omega_n(z)=\begin{bmatrix} \theta_1(z) & \theta_2(z) & \ldots &  \end{bmatrix}, \quad z\in \bB_d.
	\]
	By assumption, we see that $\theta_k\in \fa_n$ for every $k\in \bN$, so that $\ran \Omega_n\subset [\fa_n H^2_d]$. Consider the weak-$*$ closed ideal
	\[
		\fc_n=\{\phi\in \M_d: \ran M_\phi\subset \ran \Omega_n\}.
	\]
	By \cite[Theorem 2.4]{DRS2015} we infer that $[\fc_n H^2_d]=\ran \Omega_n\subset [\fa_n H^2_d]$ and thus $\fc_n\subset \fa_n$.
	Apply now Theorem \ref{T:SSS} to find a contractive multiplier $\omega_n\in \fc_n\subset \fa_n$ satisfying $|\omega_n(\lambda_n)|>\delta$.
	By definition of $\fa_n$, we see that $\omega_n(\lambda_m)=0$ for every $m\neq n$.
\end{proof}

Finally, we can state and prove the main result of this section. The reader may wish to compare it with \cite[Theorem 4.4]{clouatreSIM1}.

\begin{theorem}\label{T:siminterpequivalence}
Let $\Lambda=\{\lambda_n:n\in \bN\}\subset \bB_d$ be a sequence. Consider the following statements.
	\begin{enumerate}
		\item[\rm{(i)}] The sequence $\Lambda$ is interpolating.
		\item[\rm{(ii)}] The $d$-tuple $Z^\fa$ is similar $D=\bigoplus_{\lambda\in\Lambda}\lambda$, where $\fa=\frk{v}_0(\Lambda)$.
		\item[\rm{(iii)}] Every AC commuting row contraction $T=(T_1,\ldots,T_d)$ satisfying $\fv_0(\Lambda)= \Ann(T)$ is similar to $D$. 
		\item[\rm{(iv)}] The sequence $\Lambda$ is strongly separated. 
		\item[\rm{(v)}] The sequence $\Lambda$ is strongly separated by inner multipliers.
	\end{enumerate}
	Then, we have that {\rm (i)} $\Leftrightarrow$ {\rm (ii)} $\Leftrightarrow$ {\rm (iii)} $\Rightarrow$ {\rm (iv)} $\Leftrightarrow$ {\rm (v)}.
\end{theorem}
\begin{proof}
	(i) $\Rightarrow$ (iii): This follows immediately from Corollary \ref{C:decompkappa0}.

	(iii) $\Rightarrow$ (ii) : Obvious.
		
	(ii) $\Rightarrow$ (i): Let $(a_n)_{n=1}^\infty$ be a bounded sequence and consider the operator $A=\oplus_{n=1}^\infty a_n$, which clearly commutes with $D$. Put $\fa=\fv_0(\Lambda)$.  By \cite[Lemma 2.10]{CT2018}, we see that $\Ann(Z^\fa)=\fa$.  Thus, applying (iii) to $Z^\fa$, there is an invertible operator $X$ such that $D=XZ^{\fa}X^{-1}$.
	Hence, $X^{-1}AX$ commutes with $Z^{\fa}$.
	By \cite[Theorem 5.1]{BTV2001}, we find $\phi\in \M_d$ such that $X^{-1}AX=\phi(Z^{\fa})$, and thus
	\[
		A=\phi(XZ^{\fa}X^{-1})=\phi(D).
	\]
	This is easily seen to imply that $\phi(\lambda_n)=a_n$ for every $n\in \bN$, whence $\Lambda$ is an interpolating sequence.

	(ii) $\Rightarrow$ (iv): This follows from Theorem \ref{T:simss}.

	(iv) $\Leftrightarrow$ (v): This is Corollary \ref{C:SSS}.
\end{proof}

The reader will notice that in the univariate setting of \cite[Theorem 4.4]{clouatreSIM1}, all five statements from the previous theorem are equivalent.
In the multivariate world however, it appears to be unknown whether strongly separated sequences are necessarily interpolating.
In fact, this implication is known to fail in the setting of the Dirichlet space on the disc (see \cite{MS1994},\cite{AHMR2017}).


\section{Application: quasi-similarity of certain commuting row contractions}\label{S:QS}

In this  section, we give another application of Theorem \ref{T:Decomposition}. Indeed, we wish to use the Jordan-type decomposition obtained therein to classify certain cyclic AC commuting row contractions up to ``quasi-similarity" by means of their annihilating ideals. 

Recall that given an ideal $\fa\subset \M_d$, we put
\[
\H_\fa=H^2_d\ominus [\fa H^2_d]
\]
and
\[
Z^{\fa}=P_{\H_\fa}M_x|_{\H_\fa}.
\]
Then, $Z^{\fa}$ is an AC commuting row contraction with cyclic vector $P_{\H_\fa}1$. Our first task is to record an elementary criterion for similarity to $Z^\fa$.
\begin{lemma}
	Let $N=(N_1,\cdots,N_d)$ be a commuting nilpotent $d$-tuple and let $z\in\CC^d$. Let $\fa_0\subset \bC[x_1,\ldots,x_d]$ denote the ideal of polynomials that annihilate $zI+N$, and let $\fa\subset \M_d$ denote the ideal generated by $\fa_0$. Assume that $zI+N$ is cyclic. Then, $zI+N$ is similar to $Z^{\fa}$.
	\label{L:NilSimModel}
\end{lemma}
\begin{proof}
Assume that $N$ acts on the Hilbert space $\fH$. Because $zI+N$ is cyclic and $N$ is nilpotent, it follows that $\frk{H}$ and $\H_\fa$ are finite dimensional. If $\xi\in \fH$ is a cyclic vector for $zI+N$ then
 \[
 \fH=\{p(zI+N)\xi:p\in\CC[x_1,\cdots,x_d]\}
 \]
while
  \[
  \H_\fa=\{p(Z^{\fa})1:p\in\CC[x_1,\cdots,x_d]\}.
  \]
Let $q$ be a polynomial. Then, we have that $q(zI+N)\xi=0$ if and only if
\[
q(zI+N)p(zI+N)\xi=0, \quad p\in\CC[x_1,\cdots,x_d].
\]
Therefore, $q(zI+N)\xi=0$ if and only if $q\in \fa_0$. Likewise,   $q(Z^\fa)P_{\H_\fa}1=0$ if and only if $q(Z^\fa)=0$, which is in turn equivalent to $q\in \fa_0$ via an application of Theorem \ref{T:GleasonsTrick}. We conclude that
\[
\dim \fH=\dim \H_\fa=\dim (\bC[x_1,\ldots,x_d]/\fa_0).
\]
Furthermore, the linear map $X:\fH\to \H_\fa$ defined as
\[
X(p(zI+N)\xi)=p(Z^\fa)P_{\H_\fa}1, \quad p\in \bC[x_1,\ldots,x_d]
\]
is well-defined and injective, and thus necessarily invertible. It is readily verified that $X(zI+N)=Z^\fa X$.
\end{proof}

Before stating the quasi-similarity theorem we are after, we record another well-known fact.

\begin{lemma}
	For each positive integer $n$, let $S^{(n)}$ and $T^{(n)}$ be two similar $d$-tuples of operators. Then, the $d$-tuples
	$	\bigoplus_{n=1}^\infty S^{(n)}$ and $\bigoplus_{n=1}^\infty T^{(n)}$
	are quasi-similar.
	\label{L:SimSeqQSim}
\end{lemma}
\begin{proof}
	By assumption, for each positive integer $n$ there is an invertible operator $X_n$ with the property that $X_n S_n=T_n X_n$. It is then readily verified that the operators
	\[
	Y=\bigoplus_{n=1}^\infty \frac{1}{\|X_n\|}X_n \qand Z=\bigoplus_{n=1}^\infty \frac{1}{\|X^{-1}_n\|}X^{-1}_n
	\]
	are bounded, injective and they have dense ranges. Moreover, 
	\[
	Y\left(\bigoplus_{n=1}^\infty S^{(n)} \right)= \left(\bigoplus_{n=1}^\infty T^{(n)} \right)Y \qand Z\left(\bigoplus_{n=1}^\infty T^{(n)} \right)= \left(\bigoplus_{n=1}^\infty S^{(n)} \right)Z. \qedhere
	\]
\end{proof}

We can now prove the main result of this section, which is an application of Theorem \ref{T:Decomposition}. 

\begin{theorem}
	Let $S=(S_1,\cdots,S_d)$ and $T=(T_1,\cdots,T_d)$ be AC commuting row contractions which are both cyclic.
	Let $\Lambda\subset \bB_d$ be an interpolating sequence and let $\kappa$ be a non-negative integer.
	Assume that 
	\[
	\fv_\kappa(\Lambda)\subset \Ann(S)=\Ann(T).
	\]
	Then, $S$ is quasi-similar to $T$.
	\label{T:QSim}
\end{theorem}
\begin{proof}
Put 
	\[
	\Lambda_0=\Z_{\bB_d}(\Ann(S))=\Z_{\bB_d}(\Ann(T)).
	\]
By Theorem \ref{T:Decomposition}, for each $z\in \Lambda_0$ there are commuting nilpotent $d$-tuples $A^{(z)}$ and $B^{(z)}$ with the property that $S$ and $T$ are similar to
	\[
	\bigoplus_{z\in \Lambda_0} (zI+A^{(z)}) \qand 	\bigoplus_{z\in \Lambda_0} (zI+B^{(z)}) 
	\]
	respectively.
	Moreover, for every $z\in\Lambda_0$, the tuples $zI+A^{(z)}$ and $zI+B^{(z)}$ are AC commuting row contractions with
	\[
	\Ann(zI+A^{(z)})=\ol{\fp(\Ann(S),z)}^{w*}, \quad \Ann(zI+B^{(z)})=\ol{\fp(\Ann(T),z)}^{w*}.
	\]
	We conclude that $\Ann(zI+A^{(z)})= \Ann(zI+B^{(z)})$.
	In particular, if we denote the ideals of polynomials annihilating $zI+A^{(z)}$ and $zI+B^{(z)}$ by $\fa_z$ and $\fb_z$ respectively, then $\fa_z=\fb_z$ for every $z\in \Lambda$. 
	
	Next, projecting any cyclic vector of $\bigoplus_{z\in \Lambda_0} (zI+A^{(z)})$ onto the appropriate component yields a cyclic vector for each $d$-tuple $zI+A^{(z)}, z\in \Lambda_0$.
	Likewise, the $d$-tuple $zI+B^{(z}$ is cyclic for every $z\in \Lambda_0$.
	We may thus invoke Lemma \ref{L:NilSimModel} to see that $zI+A^{(z)}$ and $zI+B^{(z)}$ are similar for every $z\in \Lambda_0$; indeed, they are both similar to $Z^{\fa_z}=Z^{\fb_z}$.
	Finally, an application of Lemma \ref{L:SimSeqQSim} shows that $\bigoplus_{z\in \Lambda_0} (zI+A^{(z)})$ is quasi-similar to $\bigoplus_{z\in \Lambda_0} (zI+B^{(z)}) $, whence $S$ is quasi-similar to $T$.
\end{proof}

It is easily verified that if two AC commuting row contractions $S$ and $T$ are quasi-similar, then $\Ann(S)=\Ann(T)$ (see for instance \cite[Lemma 2.12 ]{CT2018}).
Furthermore, we mention that in the univariate situation, the previous theorem holds without any restriction on the annihilating ideals \cite[Theorem 2.3]{bercovici1988}.
A multivariate version of this single variable theorem can be found in \cite[Corollary 3.7]{CT2018}.
It should be noted however that at present, \cite[Corollary 3.7]{CT2018} only yields a certain one-sided version of quasi-similarity.
The appeal of Theorem \ref{T:QSim} is precisely that it fixes this shortcoming, at the cost of being more restrictive in its assumptions.

As a byproduct of the ongoing discussion, we remark that higher order vanishing ideals of a given interpolating sequence $\Lambda$ are determined by their polynomial ideals, in the following precise sense.

\begin{corollary} 
	Let $\Lambda\subset \bB_d$ be an interpolating sequence.
	Let $\frk{a}$ and $\frk{b}$ be weak-$*$ closed ideals in $\mc{M}_d$ both containing $\frk{v}_\kappa(\Lambda)$ for some non-negative integer $\kappa$, and suppose that both ideals are contained in $\frk{v}_0(\Lambda)$.
	If $\PI(\frk{a},z)=\PI(\frk{b},z)$ for every $z\in\Lambda$, then $\frk{a}=\frk{b}$.
	\label{C:SamePISameI}
\end{corollary} 
\begin{proof}
	Let $\frk{p}_z=\PI(\frk{a},z)=\PI(\frk{b},z)$ for $z\in\Lambda$. Put $\Lambda_0=\Z_{\bB_d}(\Ann(T))\subset \Lambda$.
	By Theorem \ref{T:Decomposition}, $Z^{\frk{a}}$ is similar to $\bigoplus_{z\in\Lambda_0} (zI+N^{(z)})$ and  $Z^{(\frk{b})}$ is similar to $\bigoplus_{z\in\Lambda_0}(zI+R^{(z)})$
	for some nilpotent $d$-tuples $N^{(z)}$ and $R^{(z)}$.
	These $d$-tuples satisfy
	\[ \Ann(zI+N^{(z)})=\overline{\frk{p}_z}^{w*}=\Ann(zI+R^{(z)}), \]
	and both $zI+N^{(z)}$ and $zI+R^{(z)}$ are cyclic since $Z^{\fa}$ and $Z^{\fb}$ are.
	Therefore $zI+N^{(z)}$ and $zI+R^{(z)}$ are similar for each $z\in \Lambda_0$ by Lemma \ref{L:NilSimModel}.
	We conclude from Lemma \ref{L:SimSeqQSim} that $Z^{\frk{a}}$ is quasi-similar to $Z^{\frk{b}}$, whence
	\[ \frk{a}=\Ann(Z^{\frk{a}})=\Ann(Z^{\frk{b}})=\frk{b}. \qedhere \]
\end{proof}

Naturally, one may now wonder whether Theorem \ref{T:QSim} can be improved to produce similarity between the row contractions.
For vanishing ideals of order zero, this is indeed the case. 

\begin{theorem}\label{T:simorder0}
	Let $S=(S_1,\cdots,S_d)$ and $T=(T_1,\cdots,T_d)$ be AC commuting row contractions which are both cyclic. Let $\Lambda\subset \bB_d$ be an interpolating sequence and  assume that
	\[
	\fv_0(\Lambda)\subset \Ann(S)=\Ann(T).
	\]
	Then, $S$ is similar to $T$.
\end{theorem}
\begin{proof}
This is an immediate consequence of Corollary \ref{C:decompkappa0}.
\end{proof}

For higher order vanishing ideals however, similarity cannot be achieved in general, even in the single variable setting. The following example illustrates this fact, and incidentally also shows that the closed range assumption found in \cite[Theorem 5.7]{clouatreSIM1} cannot simply be removed.

\begin{example}\label{E:SimTroubled=1}
Let $\Lambda=\{\lambda_n:n\in \bN\}$ be an infinite interpolating sequence in $\bB_1$. For each positive integer $n\geq 1$, let $0<\eps_n<1$ and consider 
\[
S_n=\begin{bmatrix} \lambda_n & 1-|\lambda_n|\\
0 & \lambda_n \end{bmatrix}, \quad T_n=\begin{bmatrix} \lambda_n & \eps_n(1-|\lambda_n|)\\
0 & \lambda_n \end{bmatrix}.
\]
Then, it is readily verified that $S_n$ and $T_n$ are AC contractions acting on $\bC^2$, with 
\[
\Ann(S_n)=\Ann(T_n)=\fv_1(\{\lambda_n\})
\]
and such that $\xi=(0,1)\in \bC^2$ is a cyclic vector.
If we let $\fH=\bigoplus_{n=1}^\infty \bC^2$, then $S=\bigoplus_{n=1}^\infty S_n$ and $T=\bigoplus_{n=1}^\infty T_n$ are AC contractions on $\fH$ with
\[
\Ann(S)=\Ann(T)=\fv_1(\Lambda).
\]
We now claim that $S$ and $T$ are cyclic. To see this, invoke Lemma \ref{L:GetIdems} to find, for each positive integer $n$, a multiplier $\theta_n\in\fv_1(\Lambda\setminus \{z_n\})$ such that $1-\theta_n\in \fv_1(\{z_n\})$. Thus, we find
\[
\theta_n(S_n)=\theta_n(T_n)=I
\]
while 
\[
\theta_m(S_n)=\theta_m(T_n)=0
\]
whenever $m\neq n$. Put 
\[
\Xi=\bigoplus_{n=1}^\infty 2^{-j}\xi\in \fH.
\]
If $p$ is a polynomial and $n\geq 1$, then we see that
\[
(p\theta_n)(S)\Xi=\frac{1}{2^n}p(S_n)\xi \qand (p\theta_n)(T)\Xi=\frac{1}{2^n}p(T_n)\xi.
\]
Using that $\xi$ is cyclic for $S_n$ and $T_n$ for every $n\geq 1$, we infer that $\Xi$ is cyclic for $S$ and $T$. Thus, $S$ and $T$ are quasi-similar by Theorem \ref{T:QSim}. 

Suppose that there is an invertible $X\in B(\frk{H})$ such that $XT=SX$. In particular, for every $n\geq 1$ we have $X\theta_n(T)=\theta_n(S)X$. But $\theta_n(S)=\theta_n(T)$ for every $n\geq 1$, and the collection $\{\theta_n(T)\}_{n=1}^\infty$ consists of pairwise orthogonal projections summing to $I$, so we see that
\[
X=\bigoplus_{n=1}^\infty X_n
\] 
where $X_n=X|_{\ran\theta_n(T)}$ for every $n\geq 1$. We conclude that
\[
X_n T_n=S_n X_n
\]
for every $n\geq 1$.  A routine calculation reveals that this forces $X_n$ to be of the form
	\[ \begin{bmatrix} a_n & b_n  \\ 0 & \eps_n a_n\end{bmatrix} \]
	for some complex numbers $a_n,b_n$. Since $X_n$ is invertible, we see that $a_n\neq 0$. Furthermore,
	\[
	\|X\| \|X^{-1}\| \geq \|X_n \| \|X_n^{-1}\|\geq \frac{|a_n|}{\eps_n |a_n|}=1/\eps_n.
	\]
Thus, if we choose the sequence $(\eps_n)$ to tend to zero, then $X$ cannot be bounded.	
\qed
\end{example}

Examples of this type can also be manufactured in several variables.
Although the argument is not much different, we provide the details so as to show how to construct AC commuting row contractions with certain prescribed annihilating ideals.

First we record a few technical facts relating to automorphisms of the ball that may be of independent interest.

\begin{lemma}\label{L:auto}
Let $T=(T_1,\ldots,T_d)$ be an AC commuting row contraction with cyclic vector $\xi$ and such that $\Ann(T)=\fv_1(\{0\})$. Let $z\in \bB_d$ and let $\Gamma:\bB_d\to\bB_d$ be an automorphism such that $\Gamma(0)=z$. Then, $\Gamma(T)$ is an AC commuting row contraction with cyclic vector $\xi$ and such that $\Ann(\Gamma(T))=\fv_1(\{z\})$.
\end{lemma}
\begin{proof}
As noted in Subsection \ref{SS:DA}, the components of $\Gamma$ lie in $\A_d$ and they form a commuting row contraction on $H^2_d$. Hence, because the $\A_d$ functional calculus is completely contractive, we see that $\Gamma(T)$ is a commuting row contraction. We note that if $(\phi_n)$ is a bounded sequence in $\A_d$ converging pointwise to $0$ on $\bB_d$, then the sequence $(\phi_n\circ \Gamma)$ has the same properties. This shows that $\Gamma(T)$ is AC if and only if $T$ is. Next, we have
\begin{align*}
	\Ann(\Gamma(T))& =\{\psi\in \M_d: \psi \circ \Gamma\in \Ann(T)\}\\
	&=\{\psi\in \M_d: \psi \circ \Gamma\in \fv_1(\{0\})\}.
	\end{align*}
	But $\Gamma(0)=z$ and $\Gamma'(0)$ is invertible since $\Gamma$ is an automorphism, so that $\psi\circ\Gamma\in \fv_1(\{0\})$ if and only if $\psi\in \fv_1(\{z\})$. We conclude that
	\[
	\Ann(\Gamma(T))=\fv_1(\{z\}).
	\]
	Finally, using that $\Gamma$ is invertible, we see that the norm closed unital algebra generated by $T_1,\ldots,T_d$ coincides with that generated by the components of $\Gamma(T)$. Therefore, $\xi$ is cyclic for $T$ if and only if it is cyclic for $\Gamma(T)$.
\end{proof}

We can now give a multivariate example showing that the conclusion of the Theorem \ref{T:QSim} cannot be improved to similarity  in general.

\begin{example}\label{E:SimTrouble}
	Define two bounded linear operators on $\bC^3$ by
	\[ N_1=\begin{bmatrix} 0 & 0 & 0 \\ 1 & 0 & 0 \\ 0 & 0 & 0 \end{bmatrix}, \quad N_2=\begin{bmatrix} 0 & 0 & 0 \\ 0 & 0 & 0 \\ 1 & 0 & 0 \end{bmatrix}. \]
	We note that 
	\[ N_1^2=N_1N_2=N_2N_1=N_2^2=0. \]
	It is readily verified that the commuting pair $N=(N_1,N_2)$ is an AC row contraction with $\Ann(N)$ generated by $\{x_1^2,x_1x_2,x_2^2\}$
	since $I,N_1,N_2$ are linearly independent. Therefore, we have $\Ann(N)=\fv_1(\{0\})$.
	For each $t>0$ we set
	\[ M_1(t)=N_1, \quad M_2(t)=N_1+t N_2 \]
	and observe that
	\[ M_1(t)M_1(t)^*+M_2(t)M_2(t)^*=\begin{bmatrix} 0 & 0 & 0 \\ 0 & 2 & t \\ 0 & t & t^2 \end{bmatrix}.\]
	If we put
	\[
	f(t)=\left\| \begin{bmatrix}  2 & t \\ t & t^2\end{bmatrix}\right\|^{1/2}=\sqrt{1+t^2/2+\sqrt{1+t^4/4}}
	\]
	then we see that
	\[ M_1(t)M_1(t)^*+M_2(t)M_2(t)^* \leq f(t)^2 I \]
	and consequently, setting
	\[
	R_1(t)=\frac{1}{f(t)}M_1(t), \quad R_2(t)=\frac{1}{f(t)}M_2(t)
	\]
	yields a commuting  row contraction $R(t)=(R_1(t),R_2(t))$.
	The pair $R(t)$ is nilpotent and hence AC.
	In fact, one readily checks that 
	\[
	R_1(t)^2=R_1(t)R_2(t)=R_2(t)^2=0
	\]
	and that $I, R_1(t),R_2(t)$ are linearly independent, so $\Ann(R(t))=\fv_1(\{0\})$ as above.
	We also note that both $N$ and $R(t)$ have $\xi=(1,0,0)$ as a cyclic vector.
Next, let $\Lambda=\{z_n:n\in \bN\}\subset \bB_2$ be an infinite interpolating sequence and let $(\eps_n)$ be a sequence of positive numbers converging to $0$. For each positive integer $n$, let $\Gamma_n:\bB_2\to\bB_2$ be an automorphism such that $\Gamma_n(0)=z_n$. Let
	\[ T=\bigoplus_{n=1}^\infty \Gamma_n(N) \qand S=\bigoplus_{n=1}^\infty \Gamma_n(R(\eps_n)), \]
	both acting on 
	\[
	\frk{H}=\bigoplus_{n=1}^\infty \CC^3.
	\]
	By Lemma \ref{L:auto}, for every $n\geq 1$ we see that $\Gamma_n(N)$ and $\Gamma_n(R(t))$ are AC commuting row contractions with cyclic vector $\xi$ and such that
	\[
	\Ann(\Gamma_n(N))=\Ann(\Gamma_n(R(\eps_n))=\fv_1(\{z_n\}).
	\]
	Thus, 
	$S$ and $T$ are AC commuting row contractions such that
 	
	\[
	\Ann(T)=\Ann(S)=\frk{v}_1(\grL).
	\]
Using Lemma \ref{L:GetIdems} and arguing exactly as in Example \ref{E:SimTroubled=1}, we see that $S$ and $T$ are cyclic, and thus, $S$ and $T$ are quasi-similar by Theorem \ref{T:QSim}. 	

Suppose that there is an invertible $X\in B(\frk{H})$ such that $XT=SX$. As in Example \ref{E:SimTroubled=1}, we see that
\[
X=\bigoplus_{n=1}^\infty X_n
\]
where
\[
X_n\Gamma_n(N)=\Gamma_n(R(\eps_n))X_n
\]
and in particular
\[
X_n N=R(\eps_n)X_n
\]
for every $n\geq 1$.  A routine calculation reveals that this forces $X_n$ to be of the form
	\[ \begin{bmatrix} a_n & 0 & 0 \\ b_n & a_nf(\eps_n)^{-1} & a_nf(\eps_n)^{-1} \\ c_n & 0 & a_n\eps_n(f(\eps_n))^{-1} \end{bmatrix} \]
	for some complex numbers $a_n,b_n,c_n$. Since $X_n$ is invertible, we see that $a_n\neq 0$.
	We compute that
	\[
	\det X_n=\frac{a_n^3\eps_n}{f(\eps_n)^2}
	\]
	whence
	\begin{align*}
	\|X\| \|X^{-1}\|&\geq \|X_n\| \|X_n^{-1}\|\geq |a_n| | \det(X_n^{-1})|^{1/3}\\
	&=\eps_n^{-1/3}f(\eps_n)^{2/3}.
	\end{align*}
	Finally, we note that
	\[
	\lim_{n\to\infty} f(\eps_n)=\sqrt{2}
	\]
	so that the previous inequality contradicts $X$ being boundedly invertible.
\qed
\end{example}

\section{Similarity of Nilpotent Tuples}\label{S:nilpotent}

Example \ref{E:SimTrouble} in the previous section showed that in general the conclusion of Theorem \ref{T:QSim} cannot be improved to similarity. Examining the construction in the example, we see that the technical difficulties boil down to obtaining \emph{norm-controlled} similarities between commuting nilpotent tuples. We investigate this question in this section. To begin, we analyze a concrete model for these tuples. We first collect some known facts in the following lemma.

\begin{lemma}\label{L:modelgauge}
	Let $\fa\subset \M_d$ be a proper ideal. Then, the following statements hold.
	\begin{enumerate}
	
	\item[\rm{(i) }] Assume that $\fa$ is generated by homogeneous polynomials. Then, for every $0\leq t\leq 2\pi$ there is a unitary operator $W_t\in B(\H_\fa)$ such that $W_t 1=1$ and
	$
		W_t Z^\fa W_t^*=e^{it}Z^{\fa}.
	$
	
		\item[\rm{(ii) }] Assume that $\fa$ is generated by monomials. Then, we have $x^\alpha\in \H_\fa$ and
	\[
		\frac{|\alpha|!}{\alpha!}\|(Z^\fa)^\alpha 1\|^2=1
	\]
	for every $\alpha\in \bN^d$ such that $x^\alpha\notin \fa$.	
	\end{enumerate}

\end{lemma}
\begin{proof}
	(i) Fix $0\leq t\leq 2\pi$. As explained in \cite[Section 3.5]{shalit2014}, there is a unitary operator $U_t\in B(H^2_d)$ such that
	\[
		(U_t f)(z)=f(e^{it}z), \quad z\in \bB_d
	\]
	for every $f\in H^2_d$. Moreover, we have that
	\[
		U_t M_{x_k}U_t^*=e^{it}M_{x_k}, \quad 1\leq k\leq d.
	\]
	Since $\fa$ is generated by homogeneous polynomials, we see that $U_t \fa U_t^*= \fa$. In particular, we obtain that $U_t \H_\fa=\H_\fa$. Hence, the operator 
	\[
	W_t=U_t|_{\H_\fa}:\H_{\fa}\to \H_{\fa}
	\]
	is unitary as well. Now, we note that $1\in \H_{\fa}$ since $\fa$ is proper and generated by homogeneous polynomials, and therefore
	\[
		W_t1=U_t1=1.
	\]
	We compute
	\begin{align*}
		W_t Z^\fa W_t^*&=U_t P_{\H_\fa}M_{x}P_{\H_\fa}U_t^*|_{\H_\fa}\\
		&=P_{\H_\fa}U_t M_{x}U_t^*|_{\H_\fa}\\
		&=e^{it}P_{\H_\fa} M_{x}|_{\H_\fa}=e^{it}Z^\fa.
	\end{align*}

	(ii) There is a subset $\F\subset \bN^d$ such that $\fa$ is generated by $\{x^\beta:\beta\in \F\}$.
	Let $\beta\in \F$.
	Since the monomials form an orthogonal basis for $H^2_d$, it is readily seen that $\langle x^\alpha, x^\beta f\rangle=0$ for all $f\in H^2_d$ unless there is $\gamma\in \bN^d$ such that $\alpha=\beta+\gamma$, which in turn implies that $x^\alpha\in \fa$.
	We conclude that $x^\alpha\in \H_\fa$ whenever $x^\alpha\notin \fa$.
	Thus, if $x^\alpha\notin \fa$ we find $(Z^\fa)^\alpha 1=x^\alpha$ and thus
	\[
		\frac{|\alpha|!}{\alpha!}\|(Z^\fa)^\alpha 1\|^2=1. \qedhere
	\]
\end{proof}

We note that property (ii) of the previous result fails without the condition that $\fa$ be generated by monomials. Indeed, let $\fa\subset \M_d$ be the weak-$*$ closed ideal generated by $x_1+x_2$. Then, we see that $x_1\notin \fa$, yet 
\[
	\langle x_1,x_1+x_2\rangle_{H^2_d}=1
\]
so that $x_1\notin \H_\fa$ and 
$
\|P_{\H_\fa} x_1\|<1.
$

Next, we show that Lemma \ref{L:modelgauge} imposes necessary conditions on an arbitrary commuting nilpotent tuple to be similar to the model.

\begin{theorem}\label{T:simnecessary}
	Let $N=(N_1,\ldots,N_d)$ be a nilpotent commuting $d$-tuple on some Hilbert space $\fH$. Let $\fa\subset \M_d$ be an ideal generated by monomials. Assume that  there is an invertible operator $X:\fH\to \H_{\fa}$ such that $XNX^{-1}=Z^\fa$.
	Then, the following statements hold.
	\begin{enumerate}
		\item[\rm{(i)}] There is a unit vector $\xi\in \fH$ which is cyclic for $N$.
		\item[\rm{(ii)}] For each $t\in [0,2\pi]$, there is an invertible operator $Y_t\in B(\fH)$ with 
		\[
		\|Y_t\|\leq \|X\| \|X^{-1}\|
		\]
		such that $Y_t^{-1}\xi=\xi$ and 
		$
			Y_t N Y_t^{-1}=e^{it}N.
		$
		\item[\rm{(iii)}] We have that
		\[
			\frac{|\alpha|!}{\alpha!} \|N^\alpha \xi\|^2\geq \frac{1}{\|X^{-1}\|\|X\|}
		\]
	for every $\alpha\in \bN^d$ such that $x^\alpha\notin \fa$.
\end{enumerate} 
\end{theorem}
\begin{proof}
	Let
	\[
		\xi=\frac{1}{\|X^{-1}1\|}X^{-1}1.
	\]
	Since $1$ is cyclic for $Z^{\fa}$, we see that the unit vector $\xi$ is cyclic for $N$ so that (i) holds. Next, invoking  Lemma \ref{L:modelgauge} we find that for every $0\leq t\leq 2\pi$ there is a unitary operator $W_t\in B(\H_\fa)$ such that $W_t 1=1$ and
	$
		W_t Z W_t^*=e^{it}Z.
	$
	Put $Y_t=X^{-1}W_tX$ for every $0\leq t\leq 2\pi$. Then, we have 
	\[
		Y_t^{-1} \xi=\frac{1}{\|X^{-1}1\|}X^{-1}W_t^*1=\xi
	\]
	and 
	\[
	\|Y_t\|\leq \|X\| \|X^{-1}\|
	\]
	for every $0\leq t\leq 2\pi$. Furthermore, we observe that
	\begin{align*}
		Y_t N Y_t^{-1}&=X^{-1}W_tX N X^{-1}W^*_tX\\
		&=X^{-1}W_tZ^\fa W^*_tX\\
		&=e^{it}X^{-1}Z^\fa X=e^{it}N.
	\end{align*}
 Hence, (ii) is satisfied. Finally, if $\alpha\in \bN^d$ has the property that $x^\alpha\notin \fa$, then by Lemma \ref{L:modelgauge} we find
	\[
		\frac{|\alpha|!}{\alpha!}\|Z^\alpha 1\|^2=1
	\]
	whence
	\begin{align*}
	\frac{|\alpha|!}{\alpha!} \|N^\alpha \xi\|^2&=\frac{|\alpha|!}{\alpha!} \|X^{-1}Z^\alpha X \xi\|^2\\
		&=\frac{|\alpha|!}{\alpha!} \frac{1}{\|X^{-1}1\|}\|X^{-1}Z^\alpha 1\|^2\\
		&\geq \frac{|\alpha|!}{\alpha!} \frac{1}{\|X^{-1}\|\|X\|}\|Z^\alpha 1\|^2= \frac{1}{\|X^{-1}\|\|X\|}
	\end{align*}
	so (iii) is established.
\end{proof}

Our next objective is to show that conditions (i),(ii) and (iii) from the previous theorem are in fact sufficient for a nilpotent commuting row contraction to be similar to the model via a similarity with controlled norm. Proving this result requires several technical lemmas.
First, we show how a norm condition can be used to control the angle between certain vectors.

\begin{lemma}\label{L:orthog}
	Let $T=(T_1,\ldots,T_d)$ be a commuting row contraction on some Hilbert space $\fH$.
	Let $\xi\in \fH$ be a unit vector.
	Let $\alpha,\beta\in\bN^d$ with $|\alpha|=|\beta|$ and let $\eps>0$.
	Assume that 
	\[
	\frac{|\alpha|!}{\alpha!}\|T^\alpha \xi\|^2\geq 1-\eps, \quad {\frac{|\beta|!}{\beta!}}\|T^\beta \xi\|^2\geq 1-\eps.
	\]
	Then we have
	\[
		\left(\frac{|\alpha|!}{\alpha!}{\frac{|\beta|!}{\beta!}}\right)^{1/2}|\langle T^\alpha \xi, T^\beta \xi \rangle|\leq \eps.
	\]
\end{lemma}
\begin{proof}
	Assume that $|\alpha|=|\beta|=\ell$ and choose $\zeta\in \bC$ with $|\zeta|=1$ such that
	\[
		\ol{\zeta}\langle T^\alpha\xi, T^\beta\xi \rangle=|\langle T^\alpha\xi, T^\beta\xi \rangle|.
	\]
	The map $\Psi_T:B(\fH)\to B(\fH)$ defined as
	\[
		\Psi_T(X)=\sum_{k=1}^d T_k X T_k^*, \quad X\in B(\fH)
	\]
	is completely positive and contractive, since $T$ is a row contraction. We see that
	\[
		\sum_{|\alpha|=\ell} \frac{\ell!}{\alpha!} T^\alpha (T^*)^\alpha=(\Psi_T\circ \Psi_T \circ\ldots\circ \Psi_T)(I)\leq I
	\]
	so in particular the pair
	$
	\left( \left(\frac{|\alpha|!}{\alpha!}\right)^{1/2}T^\alpha , \left(\frac{|\beta|!}{\beta!}\right)^{1/2}T^\beta\right)
	$
	is a row contraction.
	Thus, in view of the equality
	\[
	\left(\frac{|\alpha|!}{\alpha!}\right)^{1/2}T^\alpha \xi+\zeta \left(\frac{|\beta|!}{\beta!}\right)^{1/2}T^\beta \xi=\begin{bmatrix} \left(\frac{|\alpha|!}{\alpha!}\right)^{1/2}T^\alpha & \left(\frac{|\beta|!}{\beta!}\right)^{1/2}T^\beta\end{bmatrix}\begin{bmatrix} \xi \\ \zeta\xi \end{bmatrix}
	\]
	we infer that
	\[
		\left\|\left(\frac{|\alpha|!}{\alpha!}\right)^{1/2}T^\alpha \xi+\zeta \left(\frac{|\beta|!}{\beta!}\right)^{1/2}T^\beta \xi\right\|^2\leq 2\|\xi\|^2=2.
	\]
	On the other hand, we also have
	\begin{align*}
		&\left\|\left(\frac{|\alpha|!}{\alpha!}\right)^{1/2}T^\alpha \xi+ \zeta \left(\frac{|\beta|!}{\beta!}\right)^{1/2}T^\beta \xi\right\|^2\\
		&=\frac{|\alpha|!}{\alpha!}\|T^\alpha\xi\|^2+ {\frac{|\beta|!}{\beta!}}\|T^\beta\xi\|^2+2\left(\frac{|\alpha|!}{\alpha!}{\frac{|\beta|!}{\beta!}}\right)^{1/2}\re \left(\ol{\zeta}\langle T^\alpha\xi, T^\beta\xi \rangle \right)\\
		&=\frac{|\alpha|!}{\alpha!}\|T^\alpha\xi\|^2+ {\frac{|\beta|!}{\beta!}}\|T^\beta\xi\|^2+2\left(\frac{|\alpha|!}{\alpha!}{\frac{|\beta|!}{\beta!}}\right)^{1/2}|\langle T^\alpha\xi, T^\beta\xi \rangle|\\
		&\geq 2(1-\eps)+2\left(\frac{|\alpha|!}{\alpha!}{\frac{|\beta|!}{\beta!}}\right)^{1/2}|\langle T^\alpha\xi, T^\beta\xi \rangle|.
	\end{align*}
	Combining these two inequalities yields
	\[
		\left(\frac{|\alpha|!}{\alpha!}{\frac{|\beta|!}{\beta!}}\right)^{1/2}|\langle T^\alpha \xi, T^\beta \xi \rangle|\leq \eps. \qedhere
	\]
\end{proof}

The next step is a key estimate.

\begin{lemma}\label{L:estimatesamelength}
	Let $T=(T_1,\ldots,T_d)$ be a commuting row contraction on some Hilbert space $\fH$. Let $\xi\in \fH$ be a unit vector, let $\ell\in \bN$ and let $\eps>0$. Assume that 
		\[
		\frac{|\alpha|!}{\alpha!}\|T^\alpha \xi\|^2\geq 1-\eps
	\]
	for every $\alpha\in \bN^d$ such that $|\alpha|=\ell$.
	For every $\alpha\in \bN^d$ with $|\alpha|=\ell$, let $c_\alpha\in \bC$.
	Then, we have that
	\[
		(1-\eps\card S) \left( \sum_{|\alpha|=\ell}|c_\alpha|^2\right)\leq \left\| \sum_{|\alpha|=\ell} \left(\frac{\ell!}{\alpha!}\right)^{1/2}c_\alpha T^\alpha \xi \right\|^2
	\]

	where 
	\[
		S=\{\alpha\in \bN^d:c_\alpha\neq 0\}.
	\]
\end{lemma}
\begin{proof}
	Consider the Grammian matrix
	\[
		G=\left[\left(\frac{\ell!}{\alpha!}\frac{\ell!}{\beta!}\right)^{1/2}\langle T^\alpha \xi, T^\beta \xi \rangle \right]_{\alpha,\beta\in S}.
	\]
	Put 
	\[
		D=\diag\left\{\frac{\ell!}{\alpha!}\|T^\alpha\xi\|^2:\alpha\in S\right\}
	\]
	and
	\[
		A=G-D.
	\]
	By assumption, we see that 
	\[
		D\geq (1-\eps)I.
	\]
	Furthermore, we may invoke Lemma \ref{L:orthog} to see that every entry of $A$ has modulus at most $\eps$.
	Because $A$ has zero diagonal, it follows that
	\[
	\|A\|\leq \eps (\card S-1).
	\]
	Therefore, we obtain
	\[
		G\geq D-\|A\|I\geq  (1-\eps\card S)I
	\]
	which immediately implies the desired statement.
\end{proof}

The previous norm estimate only applies to vectors that can be obtained as linear combinations of images of powers of $T$ with the same length.
In order to move past this restriction, we need the following tool.

\begin{lemma}\label{L:projhomog}
	Let $T=(T_1,\ldots,T_d)$ be a commuting $d$-tuple on some Hilbert space $\fH$.
  Let $\xi\in \fH$ be a unit vector.
  Assume that there is a constant $\gamma>0$ such that for each $t\in [0,2\pi]$, there is an invertible operator $Y_t\in B(\fH)$ such that $\|Y_t\|\leq \gamma, Y_t^{-1}\xi=\xi$ and 
	$
		Y_t T Y_t^{-1}=e^{it}T.
	$
	Let $\Xi\subset \bN^d$ be a finite subset. Then,
	\[
	\left\|\sum_{\substack{\alpha\in \Xi\\|\alpha|=\ell}}c_\alpha T^\alpha\xi \right\|\leq  \gamma\left\|\sum_{\alpha\in \Xi}c_\alpha T^\alpha\xi \right\|
	\]
	for every $\ell\in \bN$ and every collection of complex numbers $\{ c_\alpha\in \bC: \alpha\in \Xi\}$.
\end{lemma}
\begin{proof}
	First, note that 
	\[
		Y_t T^\alpha Y_t^{-1}=e^{i|\alpha|t}T^\alpha
	\]
	for every $\alpha\in \bN^d$. Fix $\ell\in \bN$. We obtain
	\begin{align*}
		\sum_{\substack{\alpha\in \Xi\\|\alpha|=\ell}}c_\alpha T^\alpha\xi&=\frac{1}{2\pi}\int_{0}^{2\pi} \left( \sum_{\alpha\in \Xi}c_\alpha e^{i|\alpha|t}T^\alpha\xi \right)e^{-i\ell t}dt\\
		&=\frac{1}{2\pi}\int_{0}^{2\pi} \left( \sum_{\alpha\in \Xi}c_\alpha Y_tT^\alpha Y_t^{-1}\xi \right)e^{-i\ell t}dt\\
		&=\frac{1}{2\pi}\int_{0}^{2\pi}  Y_t\left( \sum_{\alpha\in \Xi}c_\alpha T^\alpha \xi \right)e^{-i\ell t}dt
	\end{align*}
	which implies that
	\begin{align*}
		\left\|\sum_{\substack{\alpha\in \Xi\\|\alpha|=\ell}}c_\alpha T^\alpha\xi \right\| &\leq  \sup_{0\leq t\leq 2\pi}\|Y_t\| \left\|\sum_{\alpha\in \Xi}c_\alpha T^\alpha \xi \right\|\\
		&\leq \gamma\left\|\sum_{\alpha\in \Xi}c_\alpha T^\alpha \xi \right\|.
	\end{align*}
\end{proof}

Gathering all our previous observations, we obtain our main technical tool.

\begin{lemma}\label{L:estimate}
	Let $N=(N_1,\ldots,N_d)$ be a nilpotent commuting row contraction on some Hilbert space $\fH$.
  Let 
	\[
		\Xi=\{\alpha\in \bN^d: N^\alpha\neq 0\}
	\]
	and let $L\in \bN$ satisfy 
	\[
		\Xi\subset \{\alpha\in \bN^d:|\alpha|\leq L\}.
	\]
	Assume that the following properties hold.
	\begin{enumerate}
		\item[\rm{(a)}] There is a unit vector $\xi\in \fH$ which is cyclic for $N$.
		\item[\rm{(b)}] There is a constant $\gamma>0$ such that for each $t\in [0,2\pi]$, there is an invertible operator $Y_t\in B(\fH)$ such that $\|Y_t\|\leq \gamma, Y_t^{-1}\xi=\xi$ and 
		$
			Y_t N Y_t^{-1}=e^{it}N.
		$
		\item[\rm{(c)}] There is $\eps>0$ such that 
		\[
			\frac{|\alpha|!}{\alpha!} \|N^\alpha \xi\|^2\geq 1-\eps
		\]
		for every $\alpha\in \Xi$.
	\end{enumerate} 
	For each $\alpha\in \Xi$, let $c_\alpha\in \bC$ and put
	\[
		h=\sum_{\alpha\in \Xi}c_\alpha \left(\frac{|\alpha|!}{\alpha!}\right)^{1/2} N^\alpha\xi.
	\]
	Then, we have
	 \[
		 \frac{(1-\eps\card \Xi)}{(L+1)\gamma^2} \left( \sum_{\alpha\in \Xi}|c_\alpha|^2\right)\leq  \|h\|^2\leq (L+1)\left( \sum_{\alpha\in \Xi}|c_\alpha|^2\right).
	 \]
\end{lemma}
\begin{proof}
	Invoking Lemma \ref{L:estimatesamelength}, for every $\ell\in \bN$ we find that
	\[
		(1-\eps\card \Xi) \left( \sum_{\substack{\alpha\in \Xi\\|\alpha|=\ell}}|c_\alpha|^2\right)\leq \left\| \sum_{\substack{\alpha\in \Xi\\|\alpha|=\ell}} \left(\frac{\ell!}{\alpha!}\right)^{1/2}c_\alpha N^\alpha \xi \right\|^2.
	\]
	Combining this inequality with Lemma \ref{L:projhomog}, we find
	\[
		(1-\eps\card \Xi) \left( \sum_{\substack{\alpha\in \Xi\\|\alpha|=\ell}}|c_\alpha|^2\right)\leq \gamma^2 \|h\|^2
	\]
	for every $\ell\in \bN$. Summing over $0\leq\ell\leq L$ we obtain
	\[
		\frac{(1-\eps\card \Xi)}{L+1} \left( \sum_{\alpha\in \Xi}|c_\alpha|^2\right)\leq  \gamma^2 \|h\|^2.
	\]
	Conversely, using that $N$ is a row contraction and arguing as in the proof of Lemma \ref{L:orthog}, we see that
	\[
		\sum_{\|\alpha|=\ell}\frac{\ell!}{\alpha!}N^\alpha N^{*\alpha}\leq I
	\]
	for every $\ell\in \bN$.
	Hence, applying the contractive row operator 
	\[
	\Gamma=\left[\left(\frac{\ell!}{\alpha!}\right)^{1/2} N^\alpha\right]_{\substack{\alpha\in \Xi\\|\alpha|=\ell}}
	\]
	to the column vector
	$
	v=[c_\alpha \xi]_{\alpha\in \Xi, |\alpha|=\ell}
	$
	we see that
	\begin{align*}
	\left\|\sum_{\substack{\alpha\in \Xi\\|\alpha|=\ell}} c_\alpha \left(\frac{\ell!}{\alpha!}\right)^{1/2} N^\alpha\xi\right\|^2&=\|\Gamma v\|^2\leq \|v\|^2= \sum_{\substack{\alpha\in \Xi\\|\alpha|=\ell}}|c_\alpha|^2.
	\end{align*}
	We infer that
		\begin{align*}
		\|h\|^2&=\left\| \sum_{\alpha\in \Xi}c_\alpha \left(\frac{|\alpha|!}{\alpha!}\right)^{1/2} N^\alpha\xi\right\|^2\leq \left( \sum_{\ell=0}^L \left\|\sum_{\substack{\alpha\in \Xi\\|\alpha|=\ell}} c_\alpha \left(\frac{\ell!}{\alpha!}\right)^{1/2} N^\alpha\xi\right\|\right)^2\\
		&\leq (L+1) \sum_{\ell=0}^L \left\|\sum_{\substack{\alpha\in \Xi\\|\alpha|=\ell}} c_\alpha \left(\frac{\ell!}{\alpha!}\right)^{1/2} N^\alpha\xi\right\|^2 
		\leq  (L+1) \sum_{\ell=0}^L \left( \sum_{\substack{\alpha\in \Xi\\|\alpha|=\ell}}|c_\alpha|^2\right)\\
		&=(L+1) \left( \sum_{\alpha\in \Xi}|c_\alpha|^2\right).
	\end{align*}
 	Thus, we have found that if 
 \[
		 h=\sum_{\alpha\in \Xi}c_\alpha \left(\frac{|\alpha|!}{\alpha!}\right)^{1/2} N^\alpha\xi
	 \]
	 then
	\[
		\frac{(1-\eps\card \Xi)}{(L+1)\gamma^2} \left( \sum_{\alpha\in \Xi}|c_\alpha|^2\right)\leq  \|h\|^2\leq (L+1) \left( \sum_{\alpha\in \Xi}|c_\alpha|^2\right). \qedhere
 	\]
\end{proof}

We can now state the main result of this section, which shows that conditions (i), (ii) and (iii) from Theorem \ref{T:simnecessary} are in fact sufficient for the existence of a norm-controlled similarity to the model.

\begin{theorem}\label{T:sim}
	Let $N=(N_1,\ldots,N_d)$ be a nilpotent commuting row contraction on some Hilbert space $\fH$. Let $\fa=\Ann(N)$ and assume that it is generated by monomials.
  Let 
	\[
		\Xi=\{\alpha\in \bN^d: N^\alpha\neq 0\}
	\]
	and let $L\in \bN$ satisfy
	\[
		\Xi\subset \{\alpha\in \bN^d:|\alpha|\leq L\}.
	\]
	Assume that the following properties hold.
	\begin{enumerate}
		\item[\rm{(a)}] There is a unit vector $\xi\in \fH$ which is cyclic for $N$.
		\item[\rm{(b)}] There is a constant $\gamma>0$ such that for each $t\in [0,2\pi]$, there is an invertible operator $Y_t\in B(\fH)$ such that $\|Y_t\|\leq \gamma, Y_t^{-1}\xi=\xi$ and 
		$
			Y_t N Y_t^{-1}=e^{it}N.
		$
		\item[\rm{(c)}] There is $0<\eps<1/\card \Xi$ such that 
		\[
			\frac{|\alpha|!}{\alpha!} \|N^\alpha \xi\|^2\geq 1-\eps
		\]
		for every $\alpha\in \Xi$.
	\end{enumerate} 
	Then, there is an invertible operator $X:\fH\to \H_{\fa}$ with 
	\[
		\|X\|\leq\frac{  (L+1)\gamma }{(1-\eps \card \Xi)^{1/2}} , \quad \|X^{-1}\|\leq L+1
	\]
	and such that $XNX^{-1}=Z^{\fa}$.
\end{theorem}
\begin{proof}
	By Lemma \ref{L:estimate}, we know that
 \[
		 \frac{(1-\eps\card \Xi)}{(L+1)\gamma^2} \left( \sum_{\alpha\in \Xi}|c_\alpha|^2\right)\leq  \left\|\sum_{\alpha\in \Xi}c_\alpha \left(\frac{|\alpha|!}{\alpha!}\right)^{1/2} N^\alpha\xi\right\|^2\leq (L+1) \left( \sum_{\alpha\in \Xi}|c_\alpha|^2\right)
 	\]
	 for every collection of complex numbers $\{c_\alpha\in \bC: \alpha\in \Xi\}$.
	 By Lemma \ref{L:modelgauge} we may apply Lemma \ref{L:estimate} to $Z^{\fa}$ as well to obtain
 \[
	 \frac{1}{L+1} \left( \sum_{\alpha\in \Xi}|c_\alpha|^2\right)\leq  \left\|\sum_{\alpha\in \Xi}c_\alpha \left(\frac{|\alpha|!}{\alpha!}\right)^{1/2} (Z^{\fa})^\alpha 1\right\|^2\leq (L+1) \left( \sum_{\alpha\in \Xi}|c_\alpha|^2\right)
 \]
	for every collection of complex numbers $\{c_\alpha\in \bC: \alpha\in \Xi\}$. Now, since $\xi$ is cyclic for $N$, we see that every element of $\fH$ can be written as
	\[
	 \sum_{\alpha\in \Xi}c_\alpha \left(\frac{|\alpha|!}{\alpha!}\right)^{1/2} N^\alpha\xi 
	\]
	for some collection of complex numbers $\{c_\alpha\in \bC: \alpha\in \Xi\}$. A similar statement holds for the vector $1\in \H_\fa$ which is cyclic for $Z^\fa$. Furthermore, because
	\[
	\fa=\Ann(N)=\Ann(Z^\fa)
	\]
	we see that 
	\[
	 \sum_{\alpha\in \Xi}c_\alpha \left(\frac{|\alpha|!}{\alpha!}\right)^{1/2} N^\alpha\xi=0
	 \]
	 is equivalent to
	 \[
	 \sum_{\alpha\in \Xi}c_\alpha \left(\frac{|\alpha|!}{\alpha!}\right)^{1/2} (Z^\fa)^\alpha 1=0.
	\]
	Consequently, we may define an invertible linear operator $X:\fH\to \H_{\fa}$ as
 	\[
 		X\left( \sum_{\alpha\in \Xi}c_\alpha \left(\frac{|\alpha|!}{\alpha!}\right)^{1/2} N^\alpha\xi\right)=\sum_{\alpha\in \Xi}c_\alpha \left(\frac{|\alpha|!}{\alpha!}\right)^{1/2} (Z^{\fa})^\alpha 1
 	\]
	 for every collection of complex numbers $\{c_\alpha\in \bC: \alpha\in \Xi\}$. 
 	Then, we have
 	\[
 		\|X\|\leq\frac{  (L+1)\gamma }{(1-\eps \card \Xi)^{1/2}} , \quad \|X^{-1}\|\leq L+1.
 	\]
 	Finally, it is clear that $XNX^{-1}=Z^\fa$.
\end{proof}

Theorem \ref{T:sim} can be used to improve the conclusion of Theorem \ref{T:QSim} to similarity in some special cases. However, we omit the resulting statement as the required assumptions make it unwieldy, and leave the details to the interested reader. Moreover, we mention that it would be interesting to obtain a refinement of Theorem \ref{T:Decomposition} in the cyclic context, in the spirit of \cite[Theorem 5.7 and Corollary 5.8]{clouatreSIM1}.
 Theorem \ref{T:sim} could provide the basis of such a refinement, but at present the required technical assumptions once again blur the picture. This may be a reflection of the fact that the world of multivariate nilpotence is much richer than its univariate counterpart.
 Indeed, even in only two variables the annihilating ideals $\langle x_1^2,x_2^2\rangle$ and $\langle x_1^2,x_1x_2,x_2^2\rangle$ can support drastically different operator theoretic properties (see \cite[Examples 4 and 5]{CT2018}).
 This stands in contrast with the relative simplicity of the single-variable nilpotent case, where Theorem \ref{T:sim} has a much sharper (and simpler) analogue \cite[Proposition 5.6]{clouatreSIM1}.

\bibliographystyle{plain}
\bibliography{biblio_spectraldecomp.bib}

\end{document}